\documentclass[reqno]{amsart}

\usepackage{amsmath,amssymb,amsthm}
\usepackage{mathtools}
\usepackage{here}
\usepackage{multirow}
\usepackage{url}
\usepackage{enumerate}
\usepackage[dvipdfmx]{hyperref}
\hypersetup{
colorlinks=true,
linkcolor=black,
citecolor=black
}

\theoremstyle{definition}
\newtheorem{theorem}{Theorem}
\newtheorem{definition}{Definition}
\newtheorem{proposition}{Proposition}
\newtheorem{lemma}{Lemma}
\newtheorem{corollary}{Corollary}
\newtheorem{remark}{Remark}
\newtheorem{example}{Example}
\newtheorem{question}{Question}
\newtheorem{conjecture}{Conjecture}

\begin{document}

\title{Determination of normalized extremal quasimodular forms of depth 1 with integral Fourier coefficients}

\author{Tomoaki Nakaya}

\address{Faculty of Mathematics, Kyushu University, 744, Motooka, Nishi-ku, Fukuoka, 819-0395, Japan}
\email{t-nakaya@math.kyushu-u.ac.jp}

\begin{abstract}
The main purpose of this paper is to determine all normalized extremal quasimodular forms of depth 1 whose Fourier coefficients are integers. 
By changing the local parameter at infinity from $q=e^{2\pi i \tau}$ to the reciprocal of the elliptic modular $j$-function, we prove that all normalized extremal quasimodular forms of depth 1 have a hypergeometric series expression and that integrality is not affected by this change of parameters.
Furthermore, by transforming these hypergeometric series expressions into a certain manageable form related to the Atkin(-like) polynomials and using the lemmas that appeared in the study of $p$-adic hypergeometric series by Dwork and Zudilin, the integrality problem can be reduced to the fact that a polynomial vanishes modulo a prime power, which we prove. 
We also prove that all extremal quasimodular forms of depth~1 with appropriate weight-dependent leading coefficients have integral Fourier coefficients by focusing on the hypergeometric expression of them. 
\end{abstract}

\maketitle

\section{Introduction} \label{sec:intro}

A quasimodular form of weight $w$ on the full modular group $\Gamma=SL_{2}(\mathbb{Z})$ is given as
\begin{align*}
\sum_{\substack{2\ell+4m+6n=w \\ \ell,m,n \geq 0}} C_{\ell,m,n} E_{2}^{\,\ell} E_{4}^{\,m} E_{6}^{\,n} \in QM_{*}(\Gamma) \coloneqq \mathbb{C}[E_{2},E_{4},E_{6}] . 
\end{align*}
Here $E_{k}=E_{k}(\tau)$ is the standard Eisenstein series on $\Gamma$ of weight $k$ defined by
\begin{align*}
E_{k}(\tau ) = 1 - \frac{2k}{B_{k}} \sum_{n=1}^{\infty} \sigma_{k-1}(n)  q^{n} \; (q=e^{2 \pi i \tau}), \quad \sigma_{k}(n) = \sum_{d \mid n} d^{k}, 
\end{align*}
where $\tau$ is a variable in the complex upper half plane $\mathfrak{H}$,  and $B_{k}$ is $k$-th Bernoulli number, e.g., $B_{2}=\tfrac{1}{6}, B_{4}= -\tfrac{1}{30}, B_{6}= \tfrac{1}{42}$. For the intrinsic definition of the quasimodular forms on a non-cocompact discrete subgroup of $SL_{2}(\mathbb{R})$, we refer to \cite[\S 5.3]{zagier2008elliptic}.
We denote the vector space of modular forms and cusp forms of weight $k$ on $\Gamma$ by $M_{k}(\Gamma)$ and $S_{k}(\Gamma)$, respectively. It is well known that $E_{k} \in M_{k}(\Gamma)$ for even $k\ge4$, but $E_{2}$ is not modular and is quasimodular. 
Any quasimodular form $f$ of weight $w$ can be uniquely written as 
\begin{align*}
f = \sum_{\ell=0}^{r} E_{2}^{\,\ell} f_{\ell}, \quad f_{\ell} \in M_{w-2\ell}(\Gamma), \quad  f_{r} \not= 0
\end{align*}
with $r \in \mathbb{Z}_{\ge0}$ and we call it depth of $f$.
Hence $E_{2}$ is the quasimodular form with minimum weight and depth on $\Gamma$. 
Let $QM_{w}^{(r)} = QM_{w}^{(r)}(\Gamma)$ denote the vector space of quasimodular forms of weight $w$ and depth $\le r$ on $\Gamma$. In particular, $QM_{w}^{(0)}(\Gamma)$ is equal to $M_{w}(\Gamma)$. (In the following, we often omit the reference to the group $\Gamma$.) 
From the fact $\dim_{\mathbb{C}} QM_{w}^{(r)} = \sum_{\ell=0}^{r} \dim_{\mathbb{C}} M_{w-2\ell}$, the generating function of the dimension of $QM_{w}^{(r)}$ is given by the following (See \cite{grabner2020quasimodular} for the explicit formula of $\dim_{\mathbb{C}} QM_{w}^{(r)}$.):
\begin{align*}
&{} \sum_{k=0}^{\infty} \dim_{\mathbb{C}} QM_{2k}^{(r)}\, T^{2k} = \frac{\sum_{\ell=0}^{r} T^{2\ell}}{(1-T^{4})(1-T^{6})} \\
&= \frac{1-T^{2(r+1)}}{(1-T^{2})(1-T^{4})(1-T^{6})} \quad (r \in \mathbb{Z}_{\ge0}, \; |T|<1 ) .
\end{align*}
In particular, since there is no quasimodular form in which the depth exceeds half the weight, 
$QM_{2k}^{(k)} = QM_{2k}^{(k+1)} = QM_{2k}^{(k+2)} = \dotsb$ holds, so by taking the limit $r \rightarrow \infty$ in the above equation, we have
\begin{align*}
\sum_{k=0}^{\infty} \dim_{\mathbb{C}} QM_{2k}^{(k)}\, T^{2k} = \frac{1}{(1-T^{2})(1-T^{4})(1-T^{6})} .
\end{align*}

We denote by $D$ the differential operator $D = \tfrac{1}{2 \pi i } \tfrac{d}{d \tau} = q \tfrac{d}{d q}$. It is well known that the Eisenstein series $E_{2},E_{4}$ and $E_{6}$ satisfy the following differential relation (equation) by Ramanujan (\cite[Thm.~0.21]{cooper2017ramanujan}. See also \cite[\S 3]{cruz2018manifold} and \cite[\S 4.1]{movasati2012quasimodular} for Halphen's contribution.):
\begin{align}
D(E_{2}) = \frac{E_{2}^{2} - E_{4}}{12}, \quad D(E_{4}) = \frac{E_{2}E_{4} - E_{6}}{3}, \quad D(E_{6}) = \frac{E_{2}E_{6} - E_{4}^{2}}{2} . \label{eq:DEisen}
\end{align}
Therefore, the ring $QM_{*}(\Gamma)$ is closed under the derivation $D$.

For $f \in \mathbb{C}[\![q]\!]$, we denote $\nu (f) = \mathrm{ord}_{q=0}(f) \in \mathbb{Z}_{\ge0} \cup \{ \infty \}$.
The notion of extremal quasimodular forms was introduced and studied by Kaneko and Koike in \cite{kaneko2006extremal}. They defined it as follows using the vanishing order $\nu(f)$ of the quasimodular form $f$. 
\begin{definition}\label{def:exqmf}
Let $f = \sum_{n=0}^{\infty} a_{n} q^{n} \in QM_{w}^{(r)} \backslash QM_{w}^{(r-1)}$ and $m= \dim_{\mathbb{C}} QM_{w}^{(r)}$. We call $f$ extremal if $\nu(f)=m-1$.
If, moreover, $a_{m-1}=1$, $f$ is said to be normalized. We denote by $G_{w}^{(r)}$ the normalized extremal quasimodular form of weight $w$ and depth $r$ on $\Gamma$ (if exists).
\end{definition}
We set
\begin{align*}
\nu_{\max}(r,w) = \max \{ n\in \mathbb{Z}_{\ge0} \text{ such that there exists } f \in QM_{w}^{(r)} \backslash \{0\} \text{ with } \nu(f) = n \} .
\end{align*}
Pellarin gave an upper bound on $\nu_{\max}(r,w)$ in \cite[Thm.~2.3]{pellarin2020extremal}, which turns out to be $\nu_{\max}(r,w)=\dim_{\mathbb{C}} QM_{w}^{(r)}-1$ for the depth $r \in \{1,2,3,4\}$.
In a recent paper \cite{grabner2020quasimodular} Grabner constructed inductively the normalized extremal quasimodular forms $G_{w}^{(r)}$ of weight $w$ and depth $r \in \{1,2,3,4\}$. More precisely, he constructed the quasimodular form $f_{w}^{(r)}$ such that
\begin{align*}
f_{w}^{(r)} = q^{m-1}(1+O(q)) \in QM_{w}^{(r)},  \quad m= \dim QM_{w}^{(r)}
\end{align*}
inductively by using the Serre derivative.
From Definition \ref{def:exqmf} of the extremal quasimodular forms, it is necessary to confirm that $f_{w}^{(r)} \in QM_{w}^{(r)} \backslash QM_{w}^{(r-1)}$, which follows from $\nu_{\max}(r,w) -\nu_{\max}(r-1,w)>0$, and hence $f_{w}^{(r)}=G_{w}^{(r)}$. 
Moreover, these extremal quasimodular forms exist uniquely for $r \in \{1,2,3,4\}$. 
This is because if the two forms $f,g = q^{m-1} +O(q^{m}) \in QM_{w}^{(r)}$ satisfy $f-g \not= 0$, then the vanishing order of $f-g \in QM_{w}^{(r)}$ is given by $\nu (f - g) \ge m > \nu_{\max}(r,w)$, which contradicts the definition of $\nu_{\max}(r,w)$.
In contrast to these cases of depth $r \le 4$, the existence and the uniqueness of $G_{w}^{(r)}$ for the depth $r\ge5$ is still open. 

Let $\mathcal{E}_{r}$ be the set of weights $w$ such that the normalized extremal quasimodular form of weight $w$ and depth $r$ on $\Gamma$ has integral Fourier coefficients. 
In \cite{kaminaka2021extremal}, based on Grabner's results, Kaminaka and Kato completely determined the sets $\mathcal{E}_{2}=\{4,8\}, \mathcal{E}_{3}=\{6\}, \mathcal{E}_{4}=\emptyset$, and showed that the set $\mathcal{E}_{1}$ is a subset of a finite set of cardinality 22. 
Our main theorem asserts that their finite set coincides with $\mathcal{E}_{1}$.
\begin{theorem} \label{thm:mainthm}
We have
\begin{align}
\mathcal{E}_{1} = \{2,6,8,10,12,14,16,18,20,22,24,28,30,32,34,38,54,58,68,80,114,118 \}.  \label{eq:setE1}
\end{align}
\end{theorem}

Note that there is no quasimodular form of weight $4$ and depth $1$, since there is no modular form of weight~2 on $\Gamma$. 
The minimum weight for the normalized extremal quasimodular forms with non-integral Fourier coefficients is 26, and the Fourier coefficients belong to $\frac{1}{5} \mathbb{Z}$:
\begin{align*}
G_{26}^{(1)} &= \frac{E_{2}(E_{4}^{6} -1640 E_{4}^{3}\Delta +269280 \Delta^{2}) - E_{4}^{2}E_{6}(E_{4}^{3}-920\Delta)}{69837768000} \\
&= q^{4} + \frac{1176}{5} q^{5} +18816 q^{6} + \cdots +316607232 q^{9} + \frac{19845981568}{5} q^{10} + O(q^{11}).
\end{align*} 
Here are some examples of the normalized extremal quasimodular forms of lower weight and depth~1. Obviously, the following forms have integral Fourier coefficients:
\begin{align*}
G_{2}^{(1)} &= E_{2} = D(\log \Delta) = 1 -24 \sum_{n=1}^{\infty} \sigma_{1}(n) q^{n}, \\
G_{6}^{(1)} &=  \frac{E_{2}E_{4} - E_{6}}{720} = \frac{D(E_{4})}{240} = \sum_{n=1}^{\infty} n \sigma_{3}(n) q^{n} , \\
G_{8}^{(1)} &= \frac{E_{4}^{2} - E_{2}E_{6}}{1008} = - \frac{D(E_{6})}{504}  = \sum_{n=1}^{\infty} n \sigma_{5}(n) q^{n} , \\ G_{10}^{(1)} &= E_{4}G_{6}^{(1)} = \frac{D(E_{8})}{480}  = \sum_{n=1}^{\infty} n \sigma_{7}(n) q^{n} ,
\end{align*}
where $\Delta$ is the ``discriminant'' cusp form of weight~12 on $\Gamma$:
\begin{align*}
\Delta (\tau ) = q \prod_{n=1}^{\infty} (1-q^{n})^{24} = \frac{E_{4}(\tau)^{3} - E_{6}(\tau)^{2}}{1728} \in S_{12}(\Gamma) .
\end{align*}
Using this infinite product expression of $\Delta(\tau)$ or the differential relation \eqref{eq:DEisen}, we see that $D(\log \Delta) = D(\Delta)/\Delta = E_{2}$. 
The first few non-trivial examples of $G_{w}^{(1)} \in \mathbb{Z}[\![q]\!]$ are given below.
\begin{align*}
G_{12}^{(1)} &= \frac{E_{4}^{3} - 1008 \Delta - E_{2}E_{4}E_{6}}{332640} = \frac{1}{2 \cdot 3 \cdot 5^{2} \cdot 7} \sum_{n=2}^{\infty} ( n\sigma_{9}(n) - \tau(n) ) q^{n}  \\
&= q^{2}+56 q^{3}+1002 q^{4}+9296 q^{5}+57708 q^{6}+269040 q^{7} +O(q^{8}),  \\
G_{14}^{(1)} &= \frac{E_{2} (E_{4}^{3} - 720 \Delta)  - E_{4}^{2}E_{6}}{393120} = \frac{1}{2 \cdot 3 \cdot 691} \sum_{n=2}^{\infty} n(\sigma_{11}(n) - \tau(n) ) q^{n} \\
&=  q^{2}+128 q^{3}+4050 q^{4}+58880 q^{5}+525300 q^{6}+3338496 q^{7} +O(q^{8}) .  
\end{align*}
In \cite[Rem.~1.2]{kaminaka2021extremal}, Grabner pointed out to Kato that the integrality of these Fourier coefficients can be proved by using a classical congruence formula between Ramanujan's tau-function $\tau(n)$, defined by $\Delta = \sum_{n=1}^{\infty} \tau(n) q^{n}$, and $\sigma_{k}(n)$ (see also Section~\ref{sec:moreint} for the integrality of $G_{14}^{(1)}$).

Since $G_{16}^{(1)} = E_{4} G_{12}^{(1)}$ holds (see Proposition \ref{prop:Gto2F1}), we also see that $G_{16}^{(1)} \in \mathbb{Z}[\![q]\!]$. 
However, it is probably difficult to express the Fourier coefficients of, say, $G_{118}^{(1)}$ in terms of known number-theoretic functions and to show congruence formulas that they satisfy. 
The key idea that avoids this difficulty is to use the reciprocal of the elliptic modular function $j(\tau)^{-1}=\Delta(\tau )/E_{4}(\tau )^{3} = q -744q^{2} +356652 q^{3} - \dotsb$ instead of $q=e^{2 \pi i \tau}$ as the local parameter at infinity (or equivalently, at the cusp of $\Gamma$). 
As will be shown later, Lemma~\ref{lem:integrality} guarantees that the integrality of the coefficients is equivalent when expanded for each local parameter. Furthermore, this idea also provides a unified method of proof for all weights by using generalized hypergeometric series.

Observing the examples of $G_{w}^{(1)}$ above, we find that its Fourier coefficients are positive, except for $G_{2}^{(1)} = E_{2}$. More generally, Kaneko and Koike conjectured in  \cite{kaneko2006extremal} that the Fourier coefficients of $G_{w}^{(r)}$ would be positive if $w\ge4$ and $1 \le r \le 4$. 
For this positivity conjecture, Grabner showed in \cite{grabner2021asymptotic} that the conjecture is true if $w \le 200$ and $1 \le r \le 4$, by proving the asymptotic formula for the Fourier coefficients of $G_{w}^{(r)}$.
Thus, from Theorem \ref{thm:mainthm}, we can conclude that the Fourier coefficients of $G_{w}^{(1)}$ are positive integers if and only if their weight $w$ belongs to $\mathcal{E}_{1}\backslash \{2\}$.

The paper is organized as follows. 
Sections \ref{sec:Gasfpsj} to \ref{sec:prfmainthm} are related to the main theorem. In Section \ref{sec:Gasfpsj} we derive that the normalized extremal quasimodular forms of depth~1 have a hypergeometric expression, and rewrite them using the Atkin-like polynomials $A_{m,a}(X)$ and their adjoint polynomials $B_{m,a}(X)$. 
Then the main theorem can be reduced to the fact that a formal power series vanishes modulo a prime power, since the polynomial part arising from the ``factorization" of a formal power series vanishes. 
To ``factorize" the formal power series, we use some results on $p$-adic hypergeometric series by Dwork and Zudilin. In Section \ref{sec:congUV}, we specialize their results and prove several propositions.
In Section \ref{sec:prfmainthm} we prove the main theorem of this paper by combining the results of the previous sections.

In Section \ref{sec:moreint} we explicitly construct extremal quasimodular forms of depth 1 with integral Fourier coefficients by focusing on their hypergeometric expression.

In Section \ref{sec:furtherdirections} we summarize some previous work on the normalized extremal quasimodular forms of depth 2 to 4 and present some supplementary results, in particular from a hypergeometric point of view. We also present some previous work in which the group $\Gamma$ is replaced by a low-level congruence subgroup or Fricke group for (extremal) quasimodular forms of depth~1.

In Appendices we give: (i) A table of the integral Fourier coefficients of the normalized extremal quasimodular forms of depth~1. (ii) Explicit formulas for the coefficients of a certain formal power series that appears in the proof of the main theorem. 

%%%%%%%%%%%%%%%%%%%%%%%%%%%%%%%%%%%%%%%%%%%%%%%%%%%%%%%%%%%%%%%%%%%%%%%%%%%%%%%%%%%%%%%%%
%%%%%%%%%%%%%%%%%%%%%%%%%%%%%%%%%%%%%%%%%%%%%%%%%%%%%%%%%%%%%%%%%%%%%%%%%%%%%%%%%%%%%%%%%
%%%%%%%%%%%%%%%%%%%%%%%%%%%%%%%%%%%%%%%%%%%%%%%%%%%%%%%%%%%%%%%%%%%%%%%%%%%%%%%%%%%%%%%%%

\section{$G_{w}^{(1)}$ as a formal power series of $j^{-1}$} \label{sec:Gasfpsj}
Let $p,q$ be non-negative integers and $a_{i},b_{j} \in \mathbb{C}$ with $b_{j}\not\in \mathbb{Z}_{\le0}$. The generalized hypergeometric series ${}_{p}F_{q}$ is defined by
\begin{align*}
{}_{p}F_{q} \left( a_{1}, \dots , a_{p} ; b_{1}, \dots , b_{q} ; z \right) = \sum_{n=0}^{\infty} \frac{(a_{1})_{n} \dotsm (a_{p})_{n}}{(b_{1})_{n} \dotsm (b_{q})_{n}} \frac{z^{n}}{n!}, 
\end{align*}
where $(a)_{0}=1, (a)_{n}= a(a+1) \cdots (a +n-1)\;(n\ge1)$ denotes the Pochhammer symbol. This series is clearly invariant under the interchange of each of the parameters $a_{i}$, and the same is true for $b_{j}$. 
When $q=p-1$, the series $F={}_{p}F_{p-1}\left( a_{1}, \dots , a_{p} ; b_{1}, \dots , b_{p-1} ; z \right)$ satisfies the differential equation
\begin{align*}
z^{p-1} (1 - z) \frac{d^{p}F}{d z^{p}} + \sum_{n=1}^{p-1} z^{n-1}(\alpha_{n}z + \beta_{n}) \frac{d^{n}F}{d z^{n}} +\alpha_{0}\,F =0,
\end{align*}
where $\alpha_{n}$ and $\beta_{n}$ are some constants that depend on the parameters $a_{i}$ and $b_{j}$. Using the Euler operator $\Theta = z \tfrac{d}{d z}$, the above differential equation can be rewritten as follows:
\begin{align*}
\{ \Theta (\Theta +b_{1} -1) \dotsm (\Theta +b_{p-1} -1) - z (\Theta + a_{1}) \dotsm (\Theta + a_{p}) \} F =0. 
\end{align*}
Here we collect some hypergeometric series identities that we will use in later discussions. 
\begin{align}
&{}_{2}F_{1} \left( \alpha, \beta ; \gamma ; z \right) = (1-z)^{\gamma-\alpha-\beta} {}_{2}F_{1} \left( \gamma -\alpha, \gamma -\beta ; \gamma ; z \right), \label{eq:Euler} \\
&{}_{2}F_{1} \left( \alpha +1, \beta ; \gamma ; z \right) = \left( 1 + \tfrac{1}{\alpha} \, z\tfrac{d}{d z} \right) {}_{2}F_{1} \left( \alpha , \beta ; \gamma ; z \right), \label{eq:dHyp} \\
\begin{split}
&{}_{2}F_{1} \left( \alpha, \beta ; \alpha+\beta+\tfrac{1}{2} ; z \right)^{2} \\
&= {}_{3}F_{2} \left( 2\alpha, \alpha+\beta , 2\beta ; 2\alpha+2\beta ,\alpha+\beta+\tfrac{1}{2} ; z \right)
\end{split} \quad \text{(Clausen's formula)}, \label{eq:Clausen} \\
\begin{split}
&{}_{2}F_{1} \left( \alpha, \beta ; \alpha +\beta -\tfrac{1}{2} ; z \right) {}_{2}F_{1} \left( \alpha, \beta -1 ; \alpha +\beta -\tfrac{1}{2} ; z \right) \\
&= {}_{3}F_{2} \left( 2\alpha, 2\beta-1, \alpha+\beta-1 ; 2\alpha +2\beta -2, \alpha +\beta -\tfrac{1}{2}  ; z \right) 
\end{split} \quad \text{(Orr's formula)}. \label{eq:Orr}
\end{align}
See \cite[pp.~85-86]{bailey1964generalized} for the proof and historical background of the equations \eqref{eq:Clausen} and \eqref{eq:Orr}.

\begin{proposition}\label{prop:EisensteinHyp}
For sufficiently large $\Im (\tau)$, we have
\begin{align}
E_{4}(\tau ) &= {}_{2}F_{1} \left( \frac{1}{12}, \frac{5}{12} ; 1 ; \frac{1728}{j(\tau)}   \right)^{4} , \quad E_{4}(\tau)^{1/2} = {}_{3}F_{2} \left( \frac{1}{6}, \frac{1}{2}, \frac{5}{6} ; 1, 1 ; \frac{1728}{j(\tau)}   \right), \label{eq:E4Hyp} \\
E_{6}(\tau ) &= \left( 1 - \frac{1728}{j(\tau)} \right)^{1/2} {}_{2}F_{1} \left( \frac{1}{12}, \frac{5}{12} ; 1 ; \frac{1728}{j(\tau)}   \right)^{6} , \label{eq:E6Hyp} \\ 
E_{2}(\tau)  &=  {}_{2}F_{1} \left( \frac{1}{12}, \frac{5}{12} ; 1 ; \frac{1728}{j(\tau)}   \right) {}_{2}F_{1} \left( -\frac{1}{12}, \frac{7}{12} ; 1 ; \frac{1728}{j(\tau)}   \right) \label{eq:E2Hyp2F1} \\
&=\left( 1 - \frac{1728}{j(\tau)} \right)^{1/2}  {}_{3}F_{2} \left( \frac{1}{2}, \frac{5}{6}, \frac{7}{6} ; 1, 1 ; \frac{1728}{j(\tau)}   \right) . \label{eq:E2Hyp3F2}
\end{align}
\end{proposition}
The hypergeometric expressions of the Eisenstein series $E_{4}$ and $E_{6}$ are classical and  well known, for the proof, see \cite{stiller1988classical}. In particular, $E_{4}^{1/4}$ is one of the solutions of a hypergeometric differential equation at $j=\infty$, and for the solution at $j=0,1728$, we refer to \cite{archinard2003exceptonal}. To obtain the second equation of \eqref{eq:E4Hyp}, we use Clausen's formula \eqref{eq:Clausen}. 
See \cite[\S 5.4]{zagier2008elliptic} for a more general explanation of the fact that a (holomorphic or meromorphic) modular form satisfies a linear differential equation with a modular function as a variable.

In contrast to $E_{4}$ and $E_{6}$, the hypergeometric expression \eqref{eq:E2Hyp2F1} of $E_{2}$ is less well known. To the best of the author's knowledge, the expression can be found in the author's Ph.D. thesis \cite[Ch.~3]{nakaya2018phdthesis} and \cite[Thm.~5]{pellarin2020some}. There is also an equivalent expression in \cite[Ch.~2.6]{movasati2017gauss}, although it looks a little different. 
For the convenience of the reader, we will briefly review the proof in  \cite{nakaya2018phdthesis} here; by calculating the logarithmic derivative of $\Delta = j^{-1} E_{4}^{3} = j^{-1} {}_{2}F_{1} ( \tfrac{1}{12}, \tfrac{5}{12} ; 1 ; \tfrac{1728}{j} )^{12}$, we have
\begin{align*}
E_{2} &= D(\log \Delta) = - \tfrac{D(j)}{j} +12\, {}_{2}F_{1} ( \tfrac{1}{12}, \tfrac{5}{12} ; 1 ; \tfrac{1728}{j} )^{-1}  D \left( {}_{2}F_{1} ( \tfrac{1}{12}, \tfrac{5}{12} ; 1 ; \tfrac{1728}{j}) \right)  \\
&= - \tfrac{D(j)}{j} +12 \, {}_{2}F_{1} ( \tfrac{1}{12}, \tfrac{5}{12} ; 1 ; \tfrac{1728}{j} )^{-1} D\left( \tfrac{1728}{j} \right) \tfrac{d}{dz} {}_{2}F_{1} ( \tfrac{1}{12}, \tfrac{5}{12} ; 1 ; z ) \big|_{z=1728/j} \\
&= - \tfrac{D(j)}{j} \, {}_{2}F_{1} ( \tfrac{1}{12}, \tfrac{5}{12} ; 1 ; \tfrac{1728}{j} )^{-1} \left\{ \left( 1 +12 z \tfrac{d}{dz} \right) {}_{2}F_{1} ( \tfrac{1}{12}, \tfrac{5}{12} ; 1 ; z ) \big|_{z=1728/j} \right\} \\
&= \tfrac{E_{6}}{E_{4}} \, {}_{2}F_{1} ( \tfrac{1}{12}, \tfrac{5}{12} ; 1 ; \tfrac{1728}{j} )^{-1} \, {}_{2}F_{1} ( \tfrac{13}{12}, \tfrac{5}{12} ; 1 ; \tfrac{1728}{j} ) \quad \text{(by \eqref{eq:dHyp})} \\
&= \tfrac{E_{6}}{E_{4}} \, {}_{2}F_{1} ( \tfrac{1}{12}, \tfrac{5}{12} ; 1 ; \tfrac{1728}{j} )^{-1} \, \left( 1-\tfrac{1728}{j} \right)^{-1/2} \, {}_{2}F_{1} ( -\tfrac{1}{12}, \tfrac{7}{12} ; 1 ; \tfrac{1728}{j} ) \quad \text{(by \eqref{eq:Euler})} \\
&= {}_{2}F_{1} ( \tfrac{1}{12}, \tfrac{5}{12} ; 1 ; \tfrac{1728}{j} ) \, {}_{2}F_{1} ( -\tfrac{1}{12}, \tfrac{7}{12} ; 1 ; \tfrac{1728}{j} ).
\end{align*}
In the fourth equality we used the fact that $D(j)=-j E_{6}/E_{4}$, which can be calculated with \eqref{eq:DEisen}.
Equation \eqref{eq:E2Hyp3F2} is obtained by transforming Equation \eqref{eq:E2Hyp2F1} using Orr's formula as $(\alpha,\beta)=(\tfrac{7}{12}, \tfrac{11}{12})$ as follows. 
\begin{align*}
E_{2} &= {}_{2}F_{1} ( \tfrac{1}{12}, \tfrac{5}{12} ; 1 ; \tfrac{1728}{j} ) \, {}_{2}F_{1} ( -\tfrac{1}{12}, \tfrac{7}{12} ; 1 ; \tfrac{1728}{j} ) \\
&= \left( 1-\tfrac{1728}{j} \right)^{1/2} \, {}_{2}F_{1} ( \tfrac{11}{12}, \tfrac{7}{12} ; 1 ; \tfrac{1728}{j} ) \, {}_{2}F_{1} ( -\tfrac{1}{12}, \tfrac{7}{12} ; 1 ; \tfrac{1728}{j} ) \\
&=  \left( 1-\tfrac{1728}{j} \right)^{1/2} \, {}_{3}F_{2} ( \tfrac{7}{6}, \tfrac{5}{6}, \tfrac{1}{2} ; 1,1 ; \tfrac{1728}{j} ).
\end{align*}
Alternatively, the expressions \eqref{eq:E2Hyp2F1} and \eqref{eq:E2Hyp3F2} can be obtained by setting $n=0$ in equations \eqref{eq:Qnt} and \eqref{eq:d2G4n+2Hyp}, since $E_{2}=G_{2}^{(1)}=G_{2}^{(2)}$.

By setting $t=1/j(\tau)$, Proposition \ref{prop:EisensteinHyp} immediately implies the following theorem, which is a generalization of Theorem 5 in \cite{stiller1988classical} for $M_{*}(\Gamma) = \mathbb{C}[E_{4},E_{6}]$.
\begin{theorem} \label{thm:isomQMF}
Put $\mathcal{F}_{1}(t) = {}_{2}F_{1} \left( \frac{1}{12}, \frac{5}{12} ; 1 ; 1728t  \right), \mathcal{F}_{2}(t) = {}_{2}F_{1} \left( -\frac{1}{12}, \frac{7}{12} ; 1 ; 1728t  \right)$, and then we have the ring isomorphism
\begin{align}
QM_{*}(\Gamma) \simeq \mathcal{F} := \mathbb{C}[ \mathcal{F}_{1}(t) \mathcal{F}_{2}(t), \mathcal{F}_{1}(t)^{4}, (1-1728t)^{1/2} \mathcal{F}_{1}(t)^{6}]. \label{eq:isomQMF}
\end{align}
Moreover, since $QM_{*}$ is closed under the derivation $D$, the ring $\mathcal{F}$ is closed under the derivation $D_{\mathcal{F}}:=(1-1728t)^{1/2}\mathcal{F}_{1}(t)^{2} \, t \frac{d}{d t}$. In other words, the isomorphism $(QM_{*}(\Gamma),D) \simeq (\mathcal{F}, D_{\mathcal{F}})$ holds as a graded differential algebra over $\mathbb{C}$.
\end{theorem}
Based on the ring isomorphism \eqref{eq:isomQMF}, we can consider the problem of the integrality of the Fourier coefficients of normalized extremal quasimodular forms as an equivalence problem in $\mathcal{F}$.

\begin{example}
The formula in $\mathcal{F}$ equivalent to the formula $D(E_{2})=\tfrac{1}{12}(E_{2}^{2}-E_{4})$ in $QM_{*}$ can be calculated as follows. 
\begin{align*}
&{} 12t \frac{d}{d t} \mathcal{F}_{1}(t) \mathcal{F}_{2}(t) =   \left( 12t \frac{d}{d t} \mathcal{F}_{1}(t) \right) \mathcal{F}_{2}(t) + \mathcal{F}_{1}(t) \left( 12t \frac{d}{d t} \mathcal{F}_{2}(t) \right)  \\
&= \mathcal{F}_{2}(t) \left( 1 +12t \frac{d}{d t} \right) \mathcal{F}_{1}(t) - \mathcal{F}_{1}(t) \left( 1 -12t \frac{d}{d t} \right) \mathcal{F}_{2}(t) \\
&= \mathcal{F}_{2}(t)\, {}_{2}F_{1} \left( \tfrac{13}{12}, \tfrac{5}{12} ; 1 ; 1728t  \right) - \mathcal{F}_{1}(t)\, {}_{2}F_{1} \left( \tfrac{11}{12}, \tfrac{7}{12} ; 1 ; 1728t  \right) \quad (\text{by \eqref{eq:dHyp}}) \\
&=  (1-1728t)^{-1/2} \left( \mathcal{F}_{2}(t)^{2} - \mathcal{F}_{1}(t)^{2} \right) \quad (\text{by \eqref{eq:Euler}}).
\end{align*}
Hence we obtain $D_{\mathcal{F}}(\mathcal{F}_{1}(t) \mathcal{F}_{2}(t)) = \tfrac{1}{12}( \mathcal{F}_{1}(t)^{2}\mathcal{F}_{2}(t)^{2} - \mathcal{F}_{1}(t)^{4})$. The equations corresponding to the remaining equations in \eqref{eq:DEisen} can be calculated in a similar way.
\end{example}

We introduce the Serre derivative (or Ramanujan--Serre derivative) $\partial_{k}$ defined by
\begin{align*}
\partial_{k} = D - \frac{k}{12} E_{2} .
\end{align*}
From this definition, it is clear that the Leibniz rule $\partial_{k+l}(f g)=\partial_{k}(f)g+f\partial_{l}(g)$ is satisfied. According to convention, we use the following symbols for the iterated Serre derivative:
\begin{align*}
\partial_{k}^{0}(f) = f, \quad \partial_{k}^{n+1}(f) = \partial_{k+2n} \circ \partial_{k}^{n}(f) \quad (n\ge0).
\end{align*}
It is well known that $\partial_{k-r} : QM_{k}^{(r)} \rightarrow QM_{k+2}^{(r)}$ for even $k$ and $r \in \mathbb{Z}_{\ge0}$, which is a special case of Proposition 3.3 in \cite{kaneko2006extremal}.
By rewriting \eqref{eq:DEisen} using the Serre derivative, we obtain the following useful consequences. 
\begin{align}
\partial_{1}(E_{2}) = -\frac{1}{12} E_{4}, \quad \partial_{4}(E_{4}) = -\frac{1}{3} E_{6}, \quad \partial_{6}(E_{6}) = -\frac{1}{2} E_{4}^{2}. \label{eq:exSerred}
\end{align}
Note that the Serre derivative is not necessarily depth-preserving, as can be seen from the first equation in \eqref{eq:exSerred}. Therefore, it is necessary to confirm the depth of the quasimodular form created by repeatedly applying a differential operator including the Serre derivative as in \eqref{eq:recGw0}.

Now we consider the following differential equation, called the Kaneko--Zagier equation or, in a more general context, a second-order modular linear differential equation: 
\begin{align}
D^{2}(f) - \frac{w}{6} E_{2} D(f) + \frac{w(w-1)}{12} D(E_{2}) f =0. \label{eq:2ndKZeqn}
\end{align}
This differential equation (with $w$ replaced by $k+1$) first appeared in the study of the $j$-invariants of supersingular elliptic curves in \cite{kaneko1998supersingular} and was characterized in \cite[\S 5]{kaneko2003modular}. Although we consider quasimodular form solutions of \eqref{eq:2ndKZeqn} in this paper, it should be emphasized that the Kaneko--Zagier equation \eqref{eq:2ndKZeqn} has modular form solutions of level $1,2,3,4$ (\cite{kaneko2003modular}), and $5$ (\cite{kaneko2006modular}) and mixed mock modular form solutions (\cite{guerzhoy2015mixed}) for appropriate $w$.

It is easy to rewrite the differential equation \eqref{eq:2ndKZeqn} as $L_{w}(f)=0$, where
\begin{align}
L_{w} \coloneqq \partial_{w-1}^{2} -\frac{w^{2}-1}{144} E_{4}  . \label{eq:dopLw}
\end{align}
We also define the differential operators $K_{w}^{\mathrm{up}}$ (\cite{grabner2020quasimodular}) and its ``adjoint'' $\widehat{K_{w}^{\mathrm{up}}}$ as follows:
\begin{align}
K_{w}^{\mathrm{up}} \coloneqq E_{4} \partial_{w-1} - \frac{w+1}{12} E_{6} , \quad \widehat{K_{w}^{\mathrm{up}}} \coloneqq E_{4} \partial_{w+3} - \frac{w+9}{12} E_{6} . \label{eq:dopKup} 
\end{align}
Then we have the following identity for the composition of the differential operators.
\begin{align}
L_{w+6} \circ K_{w}^{\mathrm{up}} = \widehat{K_{w}^{\mathrm{up}}} \circ L_{w} . \label{eq:LKup}
\end{align}
We will not give the proof, since it is a simple calculation by using Ramanujan's identity \eqref{eq:DEisen}. However, this identity is important for investigating the inductive structure of the solutions of the differential equation $L_{w}(f)=0$.
Indeed, if a function $f$ satisfies $L_{w}(f)=0$, we see that $F:=K_{w}^{\mathrm{up}}(f)$ satisfies $L_{w+6}(F) =0$. 
Therefore, the identity \eqref{eq:LKup} gives not only the structure of the quasimodular solution of \eqref{eq:2ndKZeqn}, but also that of a modular solution, a logarithmic solution, and a formal $q$-series solution.

The differential operator \eqref{eq:LKup} can be considered as a third-order differential operator, so that the given second-order differential operators $L_{w+6}$ and $L_{w}$ are in the left and right factors, respectively. 
From this point of view, the factors $K_{w}^{\mathrm{up}}$ and its adjoint are uniquely and independently determined from the Fourier coefficients of the extremal quasimodular forms, up to a constant multiple.

We define the sequence of power series $G_{k}^{*}$ by the following differential recursions:
\begin{align}
&{} G_{0}^{*} = 1,\quad G_{w+6}^{*} = \frac{w+6}{72(w+1)(w+5)} K_{w}^{\mathrm{up}}(G_{w}^{*}), \label{eq:recGw0} \\
&{} G_{w+2}^{*} = \frac{12}{w+1} \partial_{w-1} (G_{w}^{*}) , \quad G_{w+4}^{*} = E_{4} G_{w}^{*} . \label{eq:recGw24}
\end{align}
The proportionality constants appearing in the definition are chosen so that the leading coefficients of $G_{w}^{*}$ and $G_{w+2}^{*}$ are $1$ as follows. Since $L_{0}(G_{0}^{*})=L_{0}(1)=0$, we have $L_{w}(G_{w}^{*})=0$ for any $w\equiv 0 \pmod{6}$ with the aid of \eqref{eq:LKup}. Then the Frobenius ansatz (see \cite[\S 2.3]{grabner2020quasimodular}) gives
\begin{align*}
G_{w}^{*} = q^{w/6} \left\{ 1+ \frac{4w(2w-3)}{w+6} q +O(q^{2}) \right\}.
\end{align*}
By applying the differential operators $\partial_{w-1}$ and $K_{w}^{\mathrm{up}}$ act on this power series, we have
\begin{align*}
\partial_{w-1}(G_{w}^{*}) &= q^{w/6} \left( \frac{w+1}{12} + O(q) \right), \\
K_{w}^{\mathrm{up}}(G_{w}^{*}) &= q^{(w+6)/6} \left\{ \frac{72(w+1)(w+5)}{w+6} + O(q) \right\} .
\end{align*}
Note that $G_{w+4}^{*}$ can also be expressed as $\frac{12}{w-1}\partial_{w+1}(G_{w+2}^{*})$.

From this construction and the property of the Serre derivative $\partial_{w-1} : QM_{w}^{(1)} \rightarrow QM_{w+2}^{(1)}$, it is clear that $G_{k}^{*} = q^{m-1}(1+O(q))$ for $m=\dim_{\mathbb{C}} QM_{k}^{(1)}$ and $G_{k}^{*} \in QM_{k}^{(1)}$ for even integer~$k$. Therefore, from the same discussion in Section \ref{sec:intro}, the power series $G_{k}^{*}$ is nothing but the normalized extremal quasimodular form $G_{k}^{(1)}$, we have $L_{w}(G_{w}^{(1)})=0$ for $w \equiv 0\pmod{6}$.

By changing the variables $z=1728/j(\tau )$, the Kaneko--Zagier equation \eqref{eq:2ndKZeqn}, which is equivalent to $L_{w}(f)=0$, is transformed into
\begin{align*}
z(1-z) \frac{d^{2} g}{d z^{2}} + \left( -\frac{w-6}{6} + \frac{w-9}{6} z \right) \frac{d g}{d z} - \frac{(w-1)(w-5)}{144} g = 0
\end{align*}
or equivalently (Recall that $\Theta = z \tfrac{d}{d z}$.)
\begin{align*}
\left\{ \Theta \left( \Theta - \frac{w-6}{6} -1  \right) - z \left( \Theta - \frac{w-1}{12} \right) \left( \Theta - \frac{w-5}{12} \right) \right\} g =0,
\end{align*}
where $g(\tau ) = E_{4}(\tau )^{-(w-1)/4} f(\tau )$. This differential equation is a hypergeometric differential equation. 
Since we now assume $w \equiv 0 \pmod{6}$,  the solution space of the above differential equation is spanned by the power series solution
\begin{align}
z^{w/6}\, {}_{2}F_{1} \left( \frac{w+1}{12}, \frac{w+5}{12} ; \frac{w}{6} +1 ; z  \right) \label{eq:Hypz}
\end{align}
and a logarithmic solution. Therefore, since the extremal quasimodular forms contain no logarithmic terms, we obtain the hypergeometric expression of $G_{w}^{(1)}$ for $w\equiv 0 \pmod{6}$.

\begin{proposition} \label{prop:Gto2F1}
The normalized extremal quasimodular forms of even weight and depth $1$ on $\Gamma$ have the following hypergeometric expressions.
\begin{align*}
G_{6n}^{(1)}(\tau ) &= j(\tau )^{-n} {}_{2}F_{1} \left( \frac{1}{12}, \frac{5}{12} ; 1 ; \frac{1728}{j(\tau )}  \right)^{2(3n-1)} P_{n} \left( \frac{1}{j(\tau )} \right), \\
G_{6n+2}^{(1)}(\tau ) &= j(\tau )^{-n}  {}_{2}F_{1} \left( \frac{1}{12}, \frac{5}{12} ; 1 ; \frac{1728}{j(\tau )}  \right)^{6n} Q_{n} \left( \frac{1}{j(\tau )} \right), \\
G_{6n+4}^{(1)}(\tau ) &= E_{4}(\tau ) G_{6n}^{(1)}(\tau ),
\end{align*}
where
\begin{align}
P_{n}(t) &:= {}_{2}F_{1} \left( \frac{1}{12}, \frac{5}{12} ; 1 ; 1728 t  \right) \, {}_{2}F_{1} \left( \frac{6n+1}{12}, \frac{6n+5}{12} ; n+1 ; 1728 t  \right), \label{eq:Pnt} \\
Q_{n}(t) &:= {}_{2}F_{1} \left( \frac{1}{12}, \frac{5}{12} ; 1 ; 1728 t  \right) \, {}_{2}F_{1} \left( \frac{6n-1}{12}, \frac{6n+7}{12} ; n+1 ; 1728 t  \right). \label{eq:Qnt}
\end{align}
\end{proposition}

\begin{proof}
The hypergeometric expression of $G_{6n}^{(1)}(\tau)$ can be obtained by setting $z=\frac{1728}{j(\tau )}$ in \eqref{eq:Hypz} and multiplying by $1728^{-n} E_{4}^{(6n-1)/4}$. 
Similar expression of $G_{6n+2}^{(1)}(\tau)$ can be obtained by applying the Serre derivative to $G_{6n}^{(1)}(\tau)$ as follows ($w=6n$):
\begin{align*}
&{} \frac{w+1}{12} G_{w+2}^{(1)} = \partial_{w-1}(G_{w}^{(1)})  = D(G_{w}^{(1)}) - \frac{w-1}{12}E_{2}G_{w}^{(1)} \\
&= D(j^{-w/6} E_{4}^{(w-1)/4})\, {}_{2}F_{1} \left( \frac{w+1}{12} , \frac{w+5}{12} ; \frac{w}{6}+1 ; \frac{1728}{j} \right) \\
&{} \quad + j^{-w/6} E_{4}^{(w-1)/4} D\left( \frac{1728}{j} \right) \frac{d}{d z} {}_{2}F_{1} \left( \frac{w+1}{12} , \frac{w+5}{12} ; \frac{w}{6}+1 ; z \right) \Bigg|_{z=1728/j}  \\
&{} \quad - \frac{w-1}{12}E_{2}\, j^{-w/6} E_{4}^{(w-1)/4}\, {}_{2}F_{1} \left( \frac{w+1}{12} , \frac{w+5}{12} ; \frac{w}{6}+1 ; \frac{1728}{j} \right) \\
&= \frac{w+1}{12} j^{-w/6} E_{4}^{(w-5)/4}  E_{6}\; {}_{2}F_{1} \left( \frac{w+1}{12} , \frac{w+5}{12} ; \frac{w}{6}+1 ; \frac{1728}{j} \right) \\
&{} \quad + j^{-w/6} E_{4}^{(w-5)/4} E_{6}\; z \frac{d}{d z} {}_{2}F_{1} \left( \frac{w+1}{12} , \frac{w+5}{12} ; \frac{w}{6}+1 ; z \right) \Bigg|_{z=1728/j} \\
&= \frac{w+1}{12} j^{-w/6} E_{4}^{(w-5)/4} E_{6}\; {}_{2}F_{1} \left( \frac{w+13}{12} , \frac{w+5}{12} ; \frac{w}{6}+1 ; \frac{1728}{j} \right) \quad (\text{by \eqref{eq:dHyp}}) \\
&= \frac{w+1}{12} j^{-w/6} E_{4}^{(w+1)/4}\; {}_{2}F_{1} \left( \frac{w-1}{12} , \frac{w+7}{12} ; \frac{w}{6}+1 ; \frac{1728}{j} \right) \quad (\text{by \eqref{eq:Euler}}).
\end{align*}
\end{proof}

\begin{remark} \label{rem:logarithmic solution}
When $w=0$, the Kaneko--Zagier equation \eqref{eq:2ndKZeqn} becomes the equation $D^{2}(f)=0$, which has 1 and $2 \pi i \tau = \log (q)$ as independent solutions.
In general, for $w=6n\, (n \in \mathbb{Z}_{\ge0})$, the two-dimensional solution space of the differential equation \eqref{eq:2ndKZeqn} has already been discussed in \cite[\S 5]{kaneko2003modular}, and its basis is given by $G_{6n}^{(1)}(\tau)$ and
\begin{align*}
H_{6n}^{(1)}(\tau) :=
\begin{cases}
G_{12m}^{(1)}(\tau) \log (q) - \dfrac{12}{N_{m,0}} E_{4}(\tau)E_{6}(\tau)\Delta(\tau)^{m-1} A_{m,0}(j(\tau)) \\
\quad \text{ if $n=2m \, (m \in \mathbb{Z}_{\ge0})$}, \\
G_{12m+6}^{(1)}(\tau) \log (q) + \dfrac{12}{N_{m,6}} E_{4}(\tau) \Delta(\tau)^{m} A_{m,6}(j(\tau)) \\
\quad \text{ if $n=2m+1\, (m \in \mathbb{Z}_{\ge0})$},
\end{cases}
\end{align*}
where $A_{m,0}(X)$ and $A_{m,6}(X)$ are Atkin-like polynomials that appear later in this section, and the numbers $N_{m,0}$ and $N_{m,6}$ are defined by \eqref{eq:d1nf02} and \eqref{eq:d1nf68}. The solution $H_{6n}^{(1)}$ can also be obtained by the recurrence formula \eqref{eq:recGw0}, where the initial value is replaced by $H_{0}^{(1)}= \log (q)$. 
\end{remark}

It is well-known that the space $M_{k}(\Gamma)$ has a basis $\{ E_{4}^{3m+\delta} E_{6}^{\varepsilon}, E_{4}^{3m+\delta-3} E_{6}^{\varepsilon} \Delta , \dotsc , E_{6}^{\varepsilon} \Delta^{m} \}$ for $k=12m+4\delta +6\varepsilon$ with $m\in \mathbb{Z}_{\ge0}, \delta \in \{0,1,2\}, \varepsilon \in \{0,1\}$. This basis is characterized by the fact that the leading terms of its Fourier expansion are $1,q, \dotsc ,q^{m}$. 
We can construct a basis for $QM_{k}^{(1)}(\Gamma)$ with the same property using the normalized extremal quasimodular forms of depth~$1$.
\begin{proposition} \label{prop:d1basisB}
For any even integer $k\ge2$, a basis $\mathcal{B}_{k}^{(1)}$ of the space $QM_{k}^{(1)}(\Gamma)$ is given by the following set: where 
the notation is that in Theorem \ref{thm:isomQMF} and Proposition \ref{prop:Gto2F1}.
\begin{align*}
\mathcal{B}_{12m}^{(1)} &= \mathcal{F}_{1}(j^{-1})^{12m-2} \left( \left\{ j^{-2\ell } P_{2\ell }(j^{-1}) \right\}_{0 \le \ell \le m} \cup  \left\{ j^{-2\ell -1} Q_{2\ell +1}(j^{-1}) \right\}_{0 \le \ell \le m-1} \right) \\
&= \left\{ E_{4}^{3m}, E_{4}^{3m-2} G_{8}^{(1)}, E_{4}^{3m-3} G_{12}^{(1)},E_{4}^{3m-5} G_{20}^{(1)}, E_{4}^{3m-6} G_{24}^{(1)},  \dotsc ,G_{12m}^{(1)}   \right\} ,  \\
\mathcal{B}_{12m+2}^{(1)} &=  \mathcal{F}_{1}(j^{-1})^{12m} \left( \left\{ j^{-2\ell } Q_{2\ell }(j^{-1}) \right\}_{0 \le \ell \le m} \cup \left\{ j^{-2\ell -1} P_{2\ell +1}(j^{-1}) \right\}_{0 \le \ell \le m-1} \right) \\
&= \left\{ E_{4}^{3m} G_{2}^{(1)}, E_{4}^{3m-2} G_{10}^{(1)}, E_{4}^{3m-3} G_{14}^{(1)},E_{4}^{3m-5} G_{22}^{(1)}, E_{4}^{3m-6} G_{26}^{(1)},  \dotsc ,G_{12m+2}^{(1)}   \right\}, \\
\mathcal{B}_{12m+4}^{(1)} &= E_{4}\, \mathcal{B}_{12m}^{(1)} , \quad \mathcal{B}_{12m+6}^{(1)} = E_{4}\, \mathcal{B}_{12m+2}^{(1)} \cup \left\{ G_{12m+6}^{(1)} \right\} , \\
\mathcal{B}_{12m+8}^{(1)} &= E_{4}^{2}\, \mathcal{B}_{12m}^{(1)} \cup \left\{ G_{12m+8}^{(1)} \right\}, \quad \mathcal{B}_{12m+10}^{(1)} = E_{4}\, \mathcal{B}_{12m+6}^{(1)} .
\end{align*}
\end{proposition}

\begin{proof}
We give a proof only for the case of $QM_{12m}^{(1)}(\Gamma)$, the remaining cases being similar. We put $f_{0}=E_{4}^{3m} = 1+ O(q), f_{1}= E_{4}^{3m-2} G_{8}^{(1)}= q+O(q^{2}), \dotsc , f_{2m}=G_{12m}^{(1)} = q^{2m} +O(q^{2m+1})$. 
For any quasimodular form $f \in QM_{12m}^{(1)}(\Gamma)$, we can determine the coefficient $a_{r}$ such that $F := f - \sum_{r=0}^{2m} a_{r} f_{r} = O(q^{2m+1})  \in QM_{12m}^{(1)}(\Gamma)$. Since $\nu_{\max}(1,12m)=2m$, the form $F$ must be $0$, and hence $QM_{12m}^{(1)}(\Gamma)$ is spanned by $\{ f_{0}, \dotsc , f_{2m} \}$. For the linear independence property, by comparing the Fourier coefficients, we have $\sum_{r=0}^{2m} a_{r} f_{r} = 0 \Rightarrow a_{r} =0\, (0\le r \le 2m)$.
\end{proof}

\begin{example}
Case of weight $18$: The following equation holds.
\begin{align*}
\begin{pmatrix} E_{4}^{3}E_{6} \\ E_{6} \Delta \\ E_{2}E_{4}^{4} \\ E_{2} E_{4} \Delta \end{pmatrix}  =  \mathcal{F}_{1}(j^{-1})^{17} \cdot A \begin{pmatrix} {}_{2}F_{1} ( -\tfrac{1}{12}, \tfrac{7}{12} ; 1 ; \tfrac{1728}{j} ) \\ j^{-1} \, {}_{2}F_{1} ( \tfrac{7}{12}, \tfrac{11}{12} ; 2 ; \tfrac{1728}{j} ) \\ j^{-2} \, {}_{2}F_{1} ( \tfrac{11}{12}, \tfrac{19}{12} ; 3 ; \tfrac{1728}{j} ) \\ j^{-3} \, {}_{2}F_{1} ( \tfrac{19}{12}, \tfrac{23}{12} ; 4 ; \tfrac{1728}{j} ) \end{pmatrix} = A \begin{pmatrix} E_{4}^{4}G_{2}^{(1)} \\ E_{4}^{3}G_{6}^{(1)} \\ E_{4} G_{14}^{(1)} \\ G_{18}^{(1)}  \end{pmatrix},
\end{align*}
where
\begin{align*}
A= \begin{pmatrix} 1 & -720 & 0 & 0 \\ 0 & 1 & -1266 & 269280 \\ 1 & 0 & 0 & 0 \\ 0 & 1 & -546 & 0 \end{pmatrix}.
\end{align*}
\end{example}

\begin{lemma} \label{lem:integrality}
The following claims hold.
\begin{enumerate}
\item $f(t) \in \mathbb{Z}[\![t]\!] \Leftrightarrow (1-1728t)^{-1/2} f(t) \in Z[\![t]\!]$. \label{claim:f12}
\item $j(\tau)^{-1} \in q (1+q\mathbb{Z}[\![q]\!])$. \label{claim:jZq}
\item $f(\tau) \in \mathbb{Z}[\![q]\!] \Leftrightarrow f(\tau) \in \mathbb{Z}[\![j^{-1}]\!]$.  \label{claim:fqj}
\item $G_{6n}^{(1)}(\tau) \in \mathbb{Z}[\![q]\!] \Leftrightarrow G_{6n+4}^{(1)}(\tau) \in \mathbb{Z}[\![q]\!]$. \label{claim:G6nG6n4}
\item $G_{6n}^{(1)}(\tau) \in \mathbb{Z}[\![q]\!] \Leftrightarrow P_{n}(t) \in \mathbb{Z}[\![t]\!]$ and $G_{6n+2}^{(1)}(\tau) \in \mathbb{Z}[\![q]\!] \Leftrightarrow Q_{n}(t) \in \mathbb{Z}[\![t]\!]$. \label{claim:GPQ}
\end{enumerate}
\end{lemma}
\begin{proof}
\begin{enumerate}
\item It is clear that the following facts imply the claim:
\begin{align*}
&{} (1-4x)^{-1/2} = \sum_{r=0}^{\infty} \binom{2r}{r} x^{r} \in \mathbb{Z}[\![x]\!], \\
&{} (1-4x)^{1/2} \left( =  \sum_{r=0}^{\infty} \frac{1}{1-2r} \binom{2r}{r} x^{r} \right) =  (1-4x) (1-4x)^{-1/2} \in \mathbb{Z}[\![x]\!].
\end{align*}
%%%%%%%%%%%%%%%%%%%%%%%%%%%%%%%%%%%%%%%%%
\item Since $E_{4} \in 1+ q\mathbb{Z}[\![q]\!]$, $1/E_{4}^{3} \in 1+ q\mathbb{Z}[\![q]\!]$ holds, and we have $j^{-1} = \Delta / E_{4}^{3} \in q (1+q\mathbb{Z}[\![q]\!])$.
%%%%%%%%%%%%%%%%%%%%%%%%%%%%%%%%%%%%%%%%%
\item Using claim \eqref{claim:jZq}, we know that the coefficients satisfy $a_{m} \in \mathbb{Z}[ b_{0}, b_{1}, \dots , b_{m}]$ and $b_{m} \in \mathbb{Z}[ a_{0}, a_{1}, \dots , a_{m}]$ when the function $f$ is expressed in two ways as $f = a_{0} + a_{1} q +a_{2} q^{2} + \cdots = b_{0} + b_{1} j^{-1} + b_{2} j^{-2} + \cdots $. This gives us the assertion. 
%%%%%%%%%%%%%%%%%%%%%%%%%%%%%%%%%%%%%%%%%
\item  It is clear from the facts that $E_{4}, 1/E_{4} \in 1+ q\mathbb{Z}[\![q]\!]$ and $G_{6n+4}^{(1)} = E_{4} G_{6n}^{(1)}$.
%%%%%%%%%%%%%%%%%%%%%%%%%%%%%%%%%%%%%%%%%
\item By combining \eqref{eq:E4Hyp}, \eqref{eq:ur} and \eqref{eq:U(t)}, we have $E_{4}^{1/2} \in 1+j^{-1} \mathbb{Z}[\![j^{-1}]\!]$ and hence $E_{4}^{1/2} \in 1+q\mathbb{Z}[\![q]\!]$ and $E_{4}^{-1/2} \in 1+q\mathbb{Z}[\![q]\!]$.  
By Proposition \ref{prop:Gto2F1}, $G_{6n}^{(1)} = j^{-n} E_{4}^{(3n-1)/2} P_{n}(j^{-1})$ holds, so if $G_{6n} \in \mathbb{Z}[\![q]\!]$, then $E_{4}^{-(3n-1)/2} G_{6n}^{(1)} \in \mathbb{Z}[\![q]\!]$ and hence $j^{-n} P_{n}(j^{-1}) \in \mathbb{Z}[\![j^{-1}]\!]$ (note \eqref{claim:fqj}). And conversely, if $P_{n}(t) \in \mathbb{Z}[\![t]\!]$, we have $G_{6n}= j^{-n} E_{4}^{(3n-1)/2} P_{n}(j^{-1}) \in \mathbb{Z}[\![j^{-1}]\!]$ and therefore $G_{6n}^{(1)} \in \mathbb{Z}[\![q]\!]$ holds.
The case of $G_{6n+2}^{(1)}$ can be proved similarly.
\end{enumerate}
\end{proof}

\begin{remark} \label{rem:E4fourthroot}
In the above proof we used that $E_{4}^{1/2} \in 1+q\mathbb{Z}[\![q]\!]$, but a stronger claim holds. The hypergeometric series ${}_{2}F_{1}$ satisfies the following quadratic transformation of Gauss:
\begin{align}
{}_{2}F_{1} \left( \alpha, \beta ; \alpha +\beta + \frac{1}{2} ; 4z(1-z) \right) = {}_{2}F_{1} \left( 2\alpha, 2\beta ; \alpha +\beta + \frac{1}{2} ; z \right) .  \label{eq:quadratictransf}
\end{align}
By putting $(\alpha, \beta, 4z(1-z))=(\tfrac{1}{12}, \tfrac{5}{12}, \tfrac{1728}{j})$ in the above equation, we have
\begin{align*}
E_{4}^{1/4} &= {}_{2}F_{1} \left( \frac{1}{12}, \frac{5}{12} ; 1 ; \frac{1728}{j}   \right) = {}_{2}F_{1} \left( \frac{1}{6}, \frac{5}{6} ; 1 ; z \right) = \sum_{n=0}^{\infty} \binom{3n}{n} \binom{6n}{3n} \left( \frac{z}{432} \right)^{n} , \\
\frac{z}{432} &= \frac{1}{864} \left\{ 1 - \left( 1- \frac{1728}{j} \right)^{1/2}  \right\} = \sum_{m=1}^{\infty} 432^{m-1} C_{m-1} j^{-m} \in j^{-1} \mathbb{Z}[\![j^{-1}]\!] , 
\end{align*}
where $C_{n} :=\tfrac{1}{n+1}\tbinom{2n}{n} = \tbinom{2n}{n}-\tbinom{2n}{n-1}$ is the $n$-th Catalan number. (The second equality holds for $n\ge1$.) Thus we have $E_{4}^{1/4} \in 1+ j^{-1}\mathbb{Z}[\![j^{-1}]\!]$ or equivalently $E_{4}^{1/4} \in 1+ q\mathbb{Z}[\![q]\!]$. 
Actually, it is known the stronger claim $E_{4}^{1/8} \in 1+ q\mathbb{Z}[\![q]\!]$ holds. See \cite{heninger2006integrality} for more details and related results.
\end{remark}

We classify each element of the set \eqref{eq:setE1} according to the remainder of modulo 12 as follows: 
\begin{align*}
\mathcal{S}_{0} &= \{ 12,24 \} , \; \mathcal{S}_{2} = \{2,14,38\} , \; \mathcal{S}_{4} =\mathcal{S}_{0} \oplus 4  = \{16,28\}, \; \mathcal{S}_{6} = \{6,18,30,54,114\}, \\ 
\mathcal{S}_{8} &= \{8,20,32,68,80\}, \;  \mathcal{S}_{10} =\mathcal{S}_{6} \oplus 4 = \{10,22,34,58,118 \}.
\end{align*}
Here the symbol $\{ \mathrm{list} \} \oplus n$ means that for each element in the list the number $n$ is added. From \eqref{claim:G6nG6n4} of Lemma \ref{lem:integrality}, it suffices to show that the function $G_{w}^{(1)}$ has integral Fourier coefficients when the weight $w$ is an element of the set $\mathcal{S}_{0} \cup \mathcal{S}_{2} \cup \mathcal{S}_{6} \cup \mathcal{S}_{8}$.

Put $P_{n}(t) = \sum_{\ell=0}^{\infty} a_{\ell}(n) t^{\ell}$ and $Q_{n}(t) = \sum_{\ell=0}^{\infty} b_{\ell}(n) t^{\ell}$. The first few coefficients are given by
\begin{align*}
a_{1}(n) &= 60 +432n + \frac{60}{n+1}, \; a_{2}(n) = 65700+305856 n+93312 n^{2} + \frac{31320}{n+1} - \frac{27720}{n+2}, \\
b_{1}(n) &= 60 +432n - \frac{84}{n+1}, \; b_{2}(n) = 3492+305856 n+93312 n^{2} - \frac{37800}{n+1} + \frac{32760}{n+2} .
\end{align*}
Consider the condition that these coefficients are integers.
In such a case, the denominator must divide the numerator of the above equations, and so $a_{1}(n)$ and $a_{2}(n)$ are integers if $n \in \{ 0,1,2,3,4,5,9,19 \}$. Similarly, if $n \in \{ 0,1,2,3,5,6,11,13 \}$, $b_{1}(n)$ and $b_{2}(n)$ are integers. 
In these lists of $n$, we exclude the case of $P_{0}(t)$, because it corresponds to the trivial case $G_{0}^{(1)}=1$.
From \eqref{claim:GPQ} of Lemma \ref{lem:integrality}, in order to prove the main theorem, we should show that the power series $P_{n}(t)$ and $Q_{n}(t)$ are actually have integral coefficients for these lists of $n$. That is, we prove that
\begin{itemize}
\item $n \in \{ 1,2,3,4,5,9,19 \} \Rightarrow P_{n}(t) \in \mathbb{Z}[\![t]\!]$, and then $G_{w}^{(1)} \in \mathbb{Z}[\![q]\!]$ for $w \in \mathcal{S}_{0} \cup \mathcal{S}_{6}$,
\item $n \in \{ 0,1,2,3,5,6,11,13 \} \Rightarrow Q_{n}(t) \in \mathbb{Z}[\![t]\!]$, and then $G_{w}^{(1)} \in \mathbb{Z}[\![q]\!]$ for $w \in \mathcal{S}_{2} \cup \mathcal{S}_{8}$.
\end{itemize}
To prove these assertions, we will now rewrite the formal power series $P_{n}(t)$ and $Q_{n}(t)$ into a more manageable form.

Although the symbols are slightly different, in their paper \cite{kaneko2006extremal}, Kaneko and Koike expressed the normalized extremal quasimodular forms by using the monic polynomials $A_{m,a}(X)$ and $B_{m,a}(X)$ as follows:
\begin{align}
&{} G_{12m}^{(1)} = \frac{1}{N_{m,0}} \left( - E_{2} E_{4} E_{6} \Delta^{m-1} A_{m,0}(j) +  \Delta^{m} B_{m,0}(j)  \right), \label{eq:G12mAB} \\
&{} G_{12m+2}^{(1)} = \frac{1}{N_{m,2}} \left( E_{2} \Delta^{m} A_{m,2}(j) - E_{4}^{2} E_{6} \Delta^{m-1} B_{m,2}(j)  \right), \label{eq:G12m2AB}\\
&{} G_{12m+6}^{(1)} = \frac{1}{N_{m,6}} \left( E_{2} E_{4} \Delta^{m} A_{m,6}(j) - E_{6} \Delta^{m} B_{m,6}(j)  \right), \label{eq:G12m6AB}\\
&{} G_{12m+8}^{(1)} = \frac{1}{N_{m,8}} \left( - E_{2} E_{6} \Delta^{m} A_{m,8}(j) + E_{4}^{2} \Delta^{m} B_{m,8}(j)  \right) ,  \label{eq:G12m8AB}
\end{align}
where the normalizing factor $N_{m,a}$ is given below for any $m\ge0$ except for $N_{0,0}=N_{0,2} \coloneqq 1$. 
\begin{align}
N_{m,0} &= 24m \binom{6m}{2m} \binom{12m}{6m} ,\quad N_{m,2} = \frac{12m+1}{12m-1} N_{m,0} , \label{eq:d1nf02} \\
N_{m,6} &= N_{m+1/2,0} = 12(2m+1) \binom{6m+3}{2m+1} \binom{12m+6}{6m+3}, \quad N_{m,8} = \frac{12m+7}{12m+5} N_{m,6}   . \label{eq:d1nf68}
\end{align}
Here, $N_{m,2}$ and $N_{m,8}$ are integers, since $\tfrac{1}{1-2n} \tbinom{2n}{n} \in \mathbb{Z}$ from the proof of \eqref{claim:f12} of Corollary \ref{lem:integrality}.
In particular, the polynomial $A_{n,2}(X)$ is equal to the original Atkin polynomial $A_{n}(X)$ treated in \cite{kaneko1998supersingular}.
We will refer to the polynomials $A_{m,a}(X)$ as the Atkin-like polynomials. 
Note that we have the polynomial corresponding to the quasimodular part as ``$A$'', which is different between us and \cite{kaneko2006extremal}. 
The reason we do not follow the symbols of \cite{kaneko1998supersingular} is to write the best rational function approximation in a unified way, which we will discuss shortly afterwards.
For the convenience of the reader, we provide below a comparison table of these symbols. In the table, polynomials in the same column are equal.
\begin{table}[H]
\caption{Comparison table of the Atkin-like polynomials and its adjoint polynomials}
\begin{center}
\begin{tabular}{|c|c|c|c|c|c|c|c|c|} \hline
\rule[-4pt]{0pt}{15pt} Our symbols & $A_{n,0}$ & $A_{n,2}$  & $A_{n,6}$ & $A_{n,8}$ & $B_{n,0}$ & $B_{n,2}$ &$B_{n,6}$ & $B_{n,8}$   \\ \hline 
\rule[-4pt]{0pt}{17pt} Symbols in \cite{kaneko2006extremal} & $B_{n}^{*}$ & $A_{n}$ & $\widetilde{A}_{n+1}$ & $\widetilde{B}_{n+1}^{*}$ & $A_{n}^{*}$ & $B_{n}$ & $\widetilde{B}_{n+1}$ & $\widetilde{A}_{n+1}^{*}$  \\ \hline
\end{tabular}
\end{center}
\end{table}

We also note that the polynomials $A_{m,a}(X)$ and $B_{m,a}(X)$ are the denominator and numerator, respectively, of the best rational-function approximation\footnote{If it is not the best approximation, it contradicts the extremality of $G_{w}^{(1)}$. As a side note, for example, to get the rational function $B_{m,8}(j)/A_{m,8}(j)$ by \texttt{Mathematica} for certain $m$, we just have to enter  PadeApproximant$[(j - 1728)f_{1}/(j f_{2}), \{ j,\text{Infinity},m \}]$, where $f_{1} = {}_{2}F_{1} \left( 13/12, 5/12 ; 1 ; 1728/j \right)$ and $f_{2}={}_{2}F_{1} \left( 1/12, 5/12 ; 1 ; 1728/j \right)$. } of the following power series. 
\begin{align*}
j(j-1728) \Phi - \frac{B_{m,0}(j)}{A_{m,0}(j)} &= - \frac{N_{m,0}\, G_{12m}^{(1)}}{\Delta^{m} A_{m,0}(j)} = - \frac{N_{m,0}}{j^{2m-1}} + O(j^{-2m}), \\
\Phi - \frac{B_{m,2}(j)}{A_{m,2}(j)} &= \frac{N_{m,2}\, G_{12m+2}^{(1)}}{E_{4}^{2}E_{6}\Delta^{m-1}A_{m,2}(j)} = \frac{N_{m,2}}{j^{2m+1}} +O(j^{-2m-2}), \\
j \Phi - \frac{B_{m,6}(j)}{A_{m,6}(j)} &= \frac{N_{m,6}\, G_{12m+6}^{(1)}}{E_{6}\Delta^{m} A_{m,6}(j)} = \frac{N_{m,6}}{j^{2m+1}} + O(j^{-2m-2}), 
\\
(j-1728) \Phi - \frac{B_{m,8}(j)}{A_{m,8}(j)} &= - \frac{N_{m,8}\, G_{12m+8}^{(1)}}{E_{4}^{2} \Delta^{m} A_{m,8}(j)} = - \frac{N_{m,8}}{j^{2m+1}} + O(j^{-2m-2}) ,
\end{align*}
where $\Phi = \Phi(j^{-1})=E_{2}E_{4}/(j E_{6})$ is the power series with $j^{-1}$ as the variable (recall Proposition~\ref{prop:EisensteinHyp}), and then $(j-1728)\Phi = E_{2}E_{6}/E_{4}^{2}$. 
More generally, for some suitable polynomial $\psi(x) \in \mathbb{Q}[x]$, the orthogonal polynomial that appears when approximating the function $j\psi(j)\Phi$ has already been considered in \cite{basha2004systems} by Basha, Getz and Nover. However, they are treated in the context of a generalization of the Atkin orthogonal polynomials, and not from the point of view of extremal quasimodular forms.

Here and throughout this paper, the symbols $u_{r}$, $U(t)$, and $V(t)$ are defined as the following positive integer and formal power series, respectively:
\begin{align}
&{} u_{r} = \frac{(6r)!}{(3r)!\, r!^{3}} = \binom{2r}{r} \binom{3r}{r} \binom{6r}{3r} = \binom{4r}{r} \binom{5r}{r} \binom{6r}{r} \quad (r\in \mathbb{Z}_{\ge0}), \label{eq:ur} \\
&{} U(t) = {}_{2}F_{1} \left( \frac{1}{12}, \frac{5}{12} ; 1 ; 1728 t \right)^{2} = {}_{3}F_{2} \left( \frac{1}{6}, \frac{1}{2}, \frac{5}{6} ; 1, 1 ; 1728 t \right) = \sum_{r=0}^{\infty} u_{r} t^{r} , \label{eq:U(t)} \\
&{} V(t) = {}_{3}F_{2} \left( \frac{1}{2}, \frac{5}{6}, \frac{7}{6} ; 1, 1 ; 1728 t \right) =\left( 1+6 t \frac{d}{d t} \right) U(t) = \sum_{r=0}^{\infty} (6r+1) u_{r} t^{r} . \label{eq:V(t)}
\end{align}
Note that by using Proposition \ref{prop:EisensteinHyp}, $U(t) = P_{1}(t)$ and $V(t)= (1-1728t)^{-1/2} Q_{0}(t)$ hold. Therefore, $G_{6}^{(1)} \in \mathbb{Z}[\![q]\!]$ and $G_{2}^{(1)} \in \mathbb{Z}[\![q]\!]$ hold with the aid of Lemma \ref{lem:integrality}.

For a given polynomial $\alpha$, we denote its reciprocal polynomial as $\widetilde{\alpha}$, i.e.,  
\begin{align*}
\alpha(j) = \sum_{k=0}^{m} c_{k} j^{k} , \quad \widetilde{\alpha}(t) := t^{m} \alpha (1/t) = \sum_{k=0}^{m} c_{k} t^{m-k} .
\end{align*}
Comparing equations \eqref{eq:G12mAB} to \eqref{eq:G12m8AB} with Proposition \ref{prop:Gto2F1} and using the hypergeometric expression of Eisenstein series in Proposition \ref{prop:EisensteinHyp}, we have
\begin{align}
&{}N_{m,0}\, t^{2m} P_{2m}(t) = - \widetilde{A_{m,0}}(t)\, (1-1728t) V(t) + \widetilde{B_{m,0}}(t) U(t)  , \label{eq:P2mAVBU}   \\ 
&{}N_{m,2} \, t^{2m} (1-1728t)^{-1/2} Q_{2m}(t) = \widetilde{A_{m,2}}(t) V(t) - \widetilde{B_{m,2}}(t) U(t) , \label{eq:Q2mAVBU} \\
&{}N_{m,6} \, t^{2m+1} (1-1728t)^{-1/2} P_{2m+1}(t) = \widetilde{A_{m,6}}(t) V(t) - \widetilde{B_{m,6}}(t) U(t)  , \label{eq:P2m1AVBU}  \\
&{}N_{m,8} \, t^{2m+1} Q_{2m+1}(t) =  - \widetilde{A_{m,8}}(t)\, (1-1728t) V(t) + \widetilde{B_{m,8}}(t) U(t) . \label{eq:Q2m1AVBU}
\end{align}
As we will see later in Section \ref{sec:prfmainthm}, the polynomials $A_{m,a}(X)$ and $B_{m,a}(X)$ are not necessarily elements of $\mathbb{Z}[X]$.
In such a case, multiply both sides of the above equation by an appropriate factor $C$ so that the right-hand side of \eqref{eq:P2mAVBU} to \eqref{eq:Q2m1AVBU} has integral coefficients.
Under this normalization, if all the coefficients of the formal power series on the right-hand side are congruent to $0$ modulo $C N_{m,a}$, then we can conclude that the corresponding formal power series $P_{n}(t)$ and $Q_{n}(t)$ have integral coefficients. Here, for $Q_{2m}(t)$ and $P_{2m+1}(t)$, note \eqref{claim:f12} in Lemma \ref{lem:integrality}.
Thus, in the next section we will investigate in detail the congruence formulas of the formal power series $U (t)$ and $V (t)$ modulo prime powers.

The above equations \eqref{eq:P2mAVBU} to \eqref{eq:Q2m1AVBU} are essentially the relations satisfied by a hypergeometric series ${}_{2}F_{1}$, and can be interpreted in the following two ways. 
\begin{enumerate}
\item Focusing on the left-hand side, from the definitions of the series $P_{n}(t)$ and $Q_{n}(t)$, the term depending on $m$ is a hypergeometric series with a certain rational parameter shifted by an integer. 
Thus, the equations \eqref{eq:P2mAVBU} to \eqref{eq:Q2m1AVBU} are concrete expressions of certain contiguous relations (or three term relations) of a hypergeometric series ${}_{2}F_{1}$.
For more general results of this kind, see \cite{ebisu2012threeterm} by Ebisu.

\item Focusing on the right-hand side, the polynomials of a certain degree or less that are multiplied by the series $U(t)$ and $V(t)$ are chosen so that the vanishing order for $t$ on the left-hand side is as large as possible. 
Therefore, such polynomials are Hermite--Pad\'{e} approximations for $U(t)$ and $V(t)$, and the left-hand side is the corresponding remainder. 
For the Hermite--Pad\'{e} approximation of generalized hypergeometric series ${}_{p}F_{p-1}$, see \cite{nesterenko1995hermite} by Nesterenko. 
\end{enumerate}
Unfortunately, the case treated in the papers \cite{ebisu2012threeterm} and \cite{nesterenko1995hermite} is that the corresponding hypergeometric differential equation has no logarithmic solution, which does not include our case (recall Remark~\ref{rem:logarithmic solution}).

%%%%%%%%%%%%%%%%%%%%%%%%%%%%%%%%%%%%%%%%%%%%%%%%%%%%%%%%%%%%%%%%%%%%%%%%%%%%%%%%%%%%%%%%%
%%%%%%%%%%%%%%%%%%%%%%%%%%%%%%%%%%%%%%%%%%%%%%%%%%%%%%%%%%%%%%%%%%%%%%%%%%%%%%%%%%%%%%%%%
%%%%%%%%%%%%%%%%%%%%%%%%%%%%%%%%%%%%%%%%%%%%%%%%%%%%%%%%%%%%%%%%%%%%%%%%%%%%%%%%%%%%%%%%%

\section{Congruence formulas for $U(t)$ and $V(t)$ modulo prime powers} \label{sec:congUV}
First, note that in the prime factorization of the normalizing factor $N_{m,a}$,  which appears in the proof of the main theorem in Section \ref{sec:prfmainthm}, the exponents of all prime factors above 11 are 1. On the other hand, under the proper normalization described at the end of the previous section, the exponents of the prime factors 2, 3, 5, and 7 are at most 8, 5, 2, and 2, respectively.
Therefore, to prove the main theorem, we need to calculate the specific congruence formulas for $U(t)$ and $V(t)$ modulo $2^{8}, 3^{5}, 5^{2}, 7^{2}$ and the appropriate prime $p \ge 11$. (Of course, results for the largest exponents are valid for smaller ones.)
More specifically, in this section we will prove that $U(t)/U(t^{p})$ and $V(t)/U(t^{p})$ modulo these prime powers give some polynomials or rational functions. 
Throughout this section, for the formal power series $X(t)$ and $Y(t)$, the symbol $X(t) \equiv Y(t) \pmod{p^{s}}$ means that all the power series coefficients of their difference $X(t)-Y(t)$ are congruent to 0 modulo $p^{s}$.

If the prime $p$ is greater than or equal to 5, the following proposition corresponds to the special case of Lemma \ref{lem:Dwork57}, but here we give a simpler alternative proof using Lucas' theorem.
\begin{proposition} \label{prop:UVLucas}
For any prime $p$, we have
\begin{align*}
U(t) &\equiv \left( \sum_{m=0}^{[p/6]} u_{m} t^{m} \right) U(t^{p}) \pmod{p}, \\
V(t) &\equiv \left( \sum_{m=0}^{[p/6]} (6m+1) u_{m} t^{m} \right) U(t^{p}) \pmod{p}.
\end{align*}
Thus, in particular, $u_{r} \equiv 0\pmod{p}$ for $r\ge1$ and $p\in \{2, 3, 5\}$.
\end{proposition}

\begin{proof}
We first recall Lucas' theorem for the binomial coefficient modulo a prime $p$; for some nonnegative integers $n,m,a$ and $b$ such that $0\le a,b \le p-1$,
\begin{align}
\binom{n p +a}{m p +b} \equiv \binom{n}{m} \binom{a}{b} \pmod{p} \label{eq:Lucas}
\end{align}
holds under the usual convention $\tbinom{0}{0}=1$, and $\tbinom{x}{y}=0$ if $x<y$. 
For some generalizations of Lucas' theorem, we refer to the extensive historical survey \cite{mestrovic2014lucas} by Me\v{s}trovi\'{c}. 

To calculate $u_{l p+m}$ for $l\ge0$ and $0 \le m \le p-1$, we classify by the value of $m$.  
From Lucas' theorem, it is easy to see that if $0\le m < p/6$, the following congruence holds:
\begin{align*}
u_{l p+m} &= \binom{2(l p+m)}{l p+m} \binom{3(l p+m)}{l p+m} \binom{6(l p+m)}{3(l p+m)} \\
&\equiv  \binom{2l}{l} \binom{3l}{l} \binom{6l}{3l} \binom{2m}{m} \binom{3m}{m} \binom{6m}{3m}  \\
&\equiv  u_{l} u_{m} \pmod{p} .
\end{align*}
Similarly, if $p/6 \le m < p/3$, since $0 \le 6m-p<p$ and $3m<p$, so
\begin{align*}
\binom{6(l p+m)}{3(l p+m)} = \binom{(6l+1)p + (6m-p)}{3l p +3m} \equiv \binom{6l+1}{3l} \binom{6m-p}{3m} \pmod{p} .
\end{align*}
Now we assume that $m < p/3$, $(6m-p)-3m=3m-p<0$ and thus the binomial coefficient $\tbinom{6m-p}{3m}$ vanishes. 
Also, if $p/3 \le m<p/2$ and $p/2 \le m<p$, $\binom{3(l p+m)}{(l p+m)}$  and $\binom{2(l p+m)}{l p+m}$ vanish, respectively. Therefore, we have
\begin{align*}
U(t) = \sum_{l=0}^{\infty} \sum_{m=0}^{p-1} u_{l p +m} t^{l p +m} \equiv \sum_{l=0}^{\infty} \sum_{m=0}^{[p/6]} u_{l} u_{m} t^{l p +m} = \left( \sum_{m=0}^{[p/6]} u_{m} t^{m} \right) U(t^{p}) \pmod{p}  .
\end{align*}
The claim for $V(t)$ is obtained from a similar calculation, noting that if $0\le m < p/6$, then $(6(l p +m)+1)u_{l p +m} \equiv (6m+1) u_{m} u_{l} \pmod{p}$. If $p\in \{2,3,5\}$ then $U(t) \equiv U(t^{p}) \equiv U(t^{p^{2}}) \equiv \dotsb \pmod{p}$ holds, we see that $U(t) \equiv 1 \pmod{p}$ and hence $u_{r} \equiv 0 \pmod{p}$ for $r\ge1$.
\end{proof}

\begin{remark} \label{rem:ur0mod8}
For $p=2$, we can prove that $u_{r} \equiv 0 \pmod{2^{3}}$. Since $E_{4} \in 1+ 240q \mathbb{Z}[\![q]\!]$, $E_{4}^{1/2} \in 1+ 120q \mathbb{Z}[\![q]\!]$ and so $E_{4}^{1/2} \in 1+ 120 j^{-1} \mathbb{Z}[\![j^{-1}]\!]$ holds. By comparing \eqref{eq:E4Hyp} and \eqref{eq:U(t)}, we have $U(t) \in 1+ 120t \mathbb{Z}[\![t]\!]$ and obtain the desired result.
\end{remark}

\begin{lemma} \label{lem:Dwork57}
Let $p\ge5$ be a prime number. Let 
\begin{align*}
F(x) = {}_{3}F_{2} \left( \frac{1}{6}, \frac{1}{2}, \frac{5}{6} ; 1, 1 ; x \right) = \sum_{m=0}^{\infty} B(m) x^{m}, \quad F_{s}(x) = \sum_{m=0}^{p^{s}-1} B(m) x^{m}
\end{align*}
and $f^{(n)}(t)=\tfrac{d^{n}}{d t^{n}}f(t)$ denotes the $n$-th derivative $(n\ge0)$. Then the following congruence formulas hold for any $s \in \mathbb{Z}_{\ge0}$:
\begin{align*}
\frac{F(x)}{F(x^{p})} \equiv \frac{F_{s+1}(x)}{F_{s}(x^{p})} \pmod{p^{s+1}}, \quad \frac{F^{(n)}(x)}{F(x)} \equiv \frac{F^{(n)}_{s+1}(x)}{F_{s+1}(x)} \pmod{p^{s+1}} .
\end{align*}
\end{lemma}

\begin{proof}
These assertions correspond to the special cases of Lemma 3.4 (i) and (ii) in \cite[p.~45]{dwork1969padic}, respectively. The lemma requires the following  three assumptions: 
\begin{enumerate}[a)]
\item $B(0)=1$.
\item $B(n+mp^{s+1})/B([n/p]+mp^{s}) \equiv B(n)/B([n/p]) \pmod{p^{s+1}}$ for all $n,m,s \in \mathbb{Z}_{\ge0}$. 
\item $B(n)/B([n/p]) \in \mathbb{Z}_{p}$ for all $n \in \mathbb{Z}_{\ge0}$.
\end{enumerate}
Note that since we now consider $F(x)$ to be the formal power series, a domain $\mathfrak{O}$ in the original assumption c) can simply be the ring of $p$-adic integers $\mathbb{Z}_{p}$.

First, the assumption a) clearly holds.
Following \cite[p.~30]{dwork1969padic}, for a given prime number $p$ and some $\theta \in \mathbb{Z}_{p} \cap \mathbb{Q}$, the symbol $\theta'$ denotes the unique element $\mathbb{Z}_{p} \cap \mathbb{Q}$ such that $p \theta' - \theta \in [0,p-1]$.
Put $(\theta_{1}, \theta_{2}, \theta_{3})=(\tfrac{1}{6}, \tfrac{1}{2}, \tfrac{5}{6})$ and then 
\begin{align*}
(\theta_{1}', \theta_{2}' , \theta_{3}' ) =
\begin{cases}
(\tfrac{1}{6}, \tfrac{1}{2}, \tfrac{5}{6}) & \text{ if } p \equiv 1 \pmod{6} \\
(\tfrac{5}{6}, \tfrac{1}{2}, \tfrac{1}{6}) & \text{ if } p \equiv 5 \pmod{6}
\end{cases}
\end{align*}
holds. Therefore, this corresponds to the case of $A(n)=B(n)$ in Corollary 2 of \cite[p.~36]{dwork1969padic},  it can be seen that assumptions b) and c) hold. 
\end{proof}

We rewrite the result of this lemma for $F(x)$ into the following statement for $U(t)$ and $V(t)$ by changing the variables.
\begin{proposition} \label{prop:UV57}
The following congruence formulas are valid.
\begin{align}
U(t) &\equiv (1+20t+10t^{2}) U(t^{5}) \pmod{5^{2}} , \label{eq:U52} \\
V(t) &\equiv (1+15t +5t^{2}) U(t^{5}) \pmod{5^{2}} , \label{eq:V52} \\
U(t) &\equiv \frac{ 1+22 t+7 t^{2}+21 t^{3} +t^{7} +36t^{8} }{ 1+t^{7} }\, U(t^{7}) \pmod{7^{2}} , \\
V(t) &\equiv \frac{ 1+7 t+42 t^{2}+7 t^{3} +43 t^{7} }{ 1+t^{7} }\, U(t^{7}) \pmod{7^{2}} .
\end{align}
\end{proposition}

\begin{proof}
Assuming that $p=5$ and $s=1$ in Lemma \ref{lem:Dwork57}, the direct calculation gives
\begin{align*}
&{} F_{2}(x) \equiv 1 +15x +15x^{2} + 15x^{5} +15x^{10} \pmod{5^2}, \\
&{} F_{1}(x^{5}) \equiv 1 +15x^{5} +15x^{10} \pmod{5^{2}}, \\
&{} (1+10x^{5}+10x^{10}) F_{1}(x^{5}) \equiv 1 \pmod{5^{2}},
\end{align*}
and so
\begin{align*}
\frac{F_{2}(x)}{F_{1}(x^{5})} \equiv (1+10x^{5}+10x^{10}) F_{2}(x) \equiv 1+15x+15x^{2} \pmod{5^{2}} .
\end{align*}
On the other hand, $u_{r} \equiv 0 \pmod{5}$ for $r\ge1$ from Proposition \ref{prop:UVLucas}, and $1728^{4r} \equiv 1  \pmod{5}$ holds, thus we have
\begin{align*}
U(t^{5}) - F(1728^{5}t^{5}) = \sum_{r=1}^{\infty} u_{r} \cdot (1-1728^{4r}) t^{5r} \equiv 0 \pmod{5^2}.
\end{align*}
Hence, with the help of Lemma \ref{lem:Dwork57}, we obtain the congruence formula \eqref{eq:U52} for $U(t)$:
\begin{align}
\frac{U(t)}{U(t^{5})} \equiv \frac{F(1728t)}{F(1728^{5}t^{5})} \equiv \frac{F_{2}(1728t)}{F_{1}(1728^{5}t^{5})} \equiv 1+20t+10t^{2} \pmod{5^{2}}.  \label{eq:UtUt5}
\end{align}
Similarly, we have
\begin{align*}
&{} (1 +10x +10x^{2} + 10x^{5} +10x^{10}) F_{2}(x) \equiv 1 \pmod{5^2} , \\
&{} \frac{d}{d x} F_{2}(x) \equiv 15 + 5x \pmod{5^{2}},
\end{align*}
and so 
\begin{align*}
\frac{\tfrac{d}{d x} F(x)}{F(x)} \equiv \frac{\tfrac{d}{d x} F_{2}(x)}{F_{2}(x)} \equiv 15 +5x \pmod{5^{2}}. 
\end{align*}
Note that changing the variables $x=1728t$, since $1728\tfrac{d}{d x} F(x) = \tfrac{d}{d t} U(t)$, we have
\begin{align}
\frac{\tfrac{d}{d t} U(t)}{U(t)} \equiv 1728 (15+5 \cdot 1728t) \equiv 20+20t \pmod{5^{2}}. \label{eq:dUtUt}
\end{align}
By combining \eqref{eq:UtUt5} and \eqref{eq:dUtUt}, we obtain the congruence formula \eqref{eq:V52} for $V(t)$:
\begin{align*}
\frac{V(t)}{U(t^{5})} &= \frac{1}{U(t^{5})}\left( 1+6 t \frac{d}{d t} \right) U(t) = \left( 1 +6 t \frac{\frac{d}{d t} U(t)}{U(t)} \right) \frac{U(t)}{U(t^{5})} \\
&\equiv (1 +6t (20+20t) ) (1+20t+10t^{2}) \equiv 1+15t +5t^{2} \pmod{5^{2}}. 
\end{align*}
Next, assuming that $p=7$ and $s=1$ in Lemma \ref{lem:Dwork57}, the direct calculation gives
\begin{align*}
&{} F_{2}(x) \equiv 1+13 x+7 x^{2}+28 x^{3}+27 x^{7}+22 x^{8} \\
&{} \qquad \qquad +7 x^{14}+42 x^{15}+28 x^{21}+21 x^{22} \pmod{7^{2}}, \\
&{} F_{1}(x^{7}) \equiv 1+13 x^{7}+7 x^{14}+28 x^{21} \equiv (1+48 x^{7}) (1+14x^{7} +21 x^{14}) \pmod{7^{2}}, \\
&{} (1+35x^{7} +28 x^{14}) F_{1}(x^{7}) \equiv 1+48 x^{7} \pmod{7^{2}},
\end{align*}
and so
\begin{align*}
\frac{F_{2}(x)}{F_{1}(x^{7})} &\equiv \frac{ (1+35x^{7} +28 x^{14}) F_{2}(x)}{ 1+48 x^{7}} \\ 
&\equiv \frac{1+13 x+7 x^{2}+28 x^{3}+13 x^{7}+36 x^{8}}{ 1+48 x^{7}}  \pmod{7^{2}} .
\end{align*}
Unlike the case of $p=5$, $U(t^{7}) - F(1728^{7}t^{7})$ is not congruent $0$ modulo $7^{2}$, so we first calculate the following congruence relation. Noting that the congruence $1728^{7} \equiv -1 \pmod{7^{2}}$,
\begin{align*}
\frac{F(-t)}{F(1728t)} &= \frac{F(1728^{7}t^{7})}{F(1728t)} \cdot \frac{F(-t)}{F(1728^{7}t^{7})} \equiv \left(\frac{F(1728t)}{F(1728^{7}t^{7})}\right)^{-1}  \frac{F(-t)}{F(-t^{7})}  \\
&\equiv  \left( \frac{F_{2}(1728t)}{F_{1}(1728^{7}t^{7})}\right)^{-1} \frac{F_{2}(-t)}{F_{1}(-t^{7})} \equiv \frac{F_{2}(-t)}{F_{2}(1728t)} \equiv \frac{1+48t^{2}}{1+35t+13t^{2}} \pmod{7^{2}} .
\end{align*}
Replacing $t$ with $t^{7}$ in the above equation, we have
\begin{align*}
\frac{F(1728^{7}t^{7})}{U(t^{7})} \equiv \frac{F(-t^{7})}{F(1728t^{7})} \equiv \frac{1+48t^{14}}{1+35t^{7}+13t^{14}} \pmod{7^{2}}
\end{align*}
and hence
\begin{align*}
\frac{U(t)}{U(t^{7})} &= \frac{F(1728^{7}t^{7})}{U(t^{7})} \cdot \frac{F(1728t)}{F(1728^{7}t^{7})} \equiv \frac{1+48t^{14}}{1+35t^{7}+13t^{14}} \cdot \frac{F_{2}(1728t)}{F_{1}(1728^{7}t^{7})} \\
&\equiv \frac{ 1+22 t+7 t^{2}+21 t^{3} +t^{7} +36t^{8} }{ 1+t^{7} } \pmod{7^{2}} .
\end{align*}
The calculation for $V(t)$ is the same as for the case of $p=5$. By substituting
\begin{align*}
\frac{\tfrac{d}{d t} U(t)}{U(t)} &\equiv 1728 \frac{\left. \tfrac{d}{d x} F_{2}(x) \middle|_{x=1728t} \right.}{F_{2}(1728t)} \\
&\equiv \frac{22+14 t+14 t^{2}+7 t^{6}+15 t^{7}+7 t^{14}+21 t^{21}}{F_{2}(1728t)} \pmod{7^{2}}
\end{align*}
into the corresponding part, we obtain
\begin{align*}
\frac{V(t)}{U(t^{7})} &= \frac{1}{U(t^{7})}\left( 1+6 t \frac{d}{d t} \right) U(t) = \left( 1 +6 t \frac{\frac{d}{d t} U(t)}{U(t)} \right) \frac{U(t)}{U(t^{7})} \\
&\equiv  \frac{1+7 t+42 t^{2}+7 t^{3}+15 t^{7}+7 t^{14}+21 t^{21}}{F_{2}(1728t)} \cdot \frac{U(t)}{U(t^{7})} \\
&\equiv \frac{ 1+7 t+42 t^{2}+7 t^{3} +43 t^{7} }{ 1+t^{7} }  \pmod{7^{2}}. 
\end{align*}
This completes the proof of the proposition.
\end{proof}

Unfortunately, since $\tfrac{1}{2} \not\in \mathbb{Z}_{2} \cap \mathbb{Q}$ and $\tfrac{1}{6}, \tfrac{5}{6}  \not\in \mathbb{Z}_{3} \cap \mathbb{Q}$, Lemma \ref{lem:Dwork57} cannot be applied directly to $p=2$ and $3$.
Therefore, we consider the following formal power series to which Lemma \ref{lem:Dwork57} can be applied:
\begin{align}
F(z) &= {}_{2}F_{1} \left( \frac{1}{6}, \frac{5}{6} ; 1 ; 432z \right) = \sum_{m=0}^{\infty} A(m) z^{m}, \quad F_{s}(z) = \sum_{m=0}^{p^{s}-1} A(m) z^{m} , \label{eq:Amzm} \\
A(m) &= A_{6}(m) = \frac{(6m)!}{(3m)!(2m)!m!} = \binom{3m}{m}\binom{6m}{3m} \quad (m \in \mathbb{Z}_{\ge0}). \label{eq:intseqAm}
\end{align}

\begin{lemma} \label{lem:Zudilin}
Suppose that the sequence of integers $A(m)$ defined by \eqref{eq:intseqAm} and let $p \in \{2,3\}$. Then for all nonnegative integers $a,b,s$, one has 
\begin{align}
\frac{A(a+b p^{s+1})}{A([a/p]+b p^{s})} \equiv \frac{A(a)}{A([a/p])} \pmod{p^{s+1}}, \quad \frac{A(a)}{A([a/p])} \in \mathbb{Z}_{p} .  \label{eq:congpropA}
\end{align}
\end{lemma}
Although it looks a slightly different, this lemma corresponds (partially) to the case of $N=6$ of Lemmas 11 and 12 in \cite{zudilin2002integrality} by Zudilin. We will not prove the above lemma again here, but will explain  below how to rewrite the calculation in \cite{zudilin2002integrality}. 
From the equation that appeared in the proof of Lemma~12 in \cite{zudilin2002integrality}, for all nonnegative integers $u,v,n,s$ such that $0\le u <p^{s}$ and $0\le v <p$, we have
\begin{align*}
\frac{A(v +up +n p^{s+1})}{A(u+n p^{s})} = \frac{A(v+up)}{A(u)} \left(1+O(p^{s+1}) \right) .
\end{align*}
By putting $n=n_{1}+n_{2},\, a=v+up+n_{1}p^{s+1} \in \mathbb{Z}_{\ge0}$ and $b=n_{2} \in \mathbb{Z}_{\ge0}$ in this equation, we obtain
\begin{align*}
\frac{A(a+b p^{s+1})}{A([a/p]+b p^{s})} \equiv \frac{A(v+up)}{A(u)} \pmod{p^{s+1}}.
\end{align*}
Since this congruence formula holds even when $b=0$, we obtain the first assertion in \eqref{eq:congpropA}. 
Next, by combining \cite[Eq.~(34)]{zudilin2002integrality} and the definition of a constant $k$ \cite[Eq.~(1)]{zudilin2002integrality} depending on $N$, we have
\begin{align*}
\mathrm{ord}_{p} \frac{A(v+mp)}{A(mp)} = \frac{k v}{p-1} \in \mathbb{Z}_{\ge 0}
\end{align*}
for all nonnegative integers $v,m$ such that $0\le v <p$. 
The second assertion in \eqref{eq:congpropA} can be obtained by rewriting $v+mp$ as $a$ in this equation. Consequently, Lemma \ref{lem:Zudilin} guarantees assumptions b) and c) of Lemma 3.4 in \cite[p.~45]{dwork1969padic}, and hence the following lemma holds.
\begin{lemma} \label{lem:Dwork23}
Let $F(z)$ and $F_{s}(z)$ be the formal power series and polynomials defined by \eqref{eq:Amzm} respectively. For $p \in \{2,3\}$, the following congruence formulas hold for any $s \in \mathbb{Z}_{\ge0}$:
\begin{align*}
\frac{F(z)}{F(z^{p})} \equiv \frac{F_{s+1}(z)}{F_{s}(z^{p})} \pmod{p^{s+1}}, \quad \frac{F^{(n)}(z)}{F(z)} \equiv \frac{F^{(n)}_{s+1}(z)}{F_{s+1}(z)} \pmod{p^{s+1}} .
\end{align*}
Here, $F^{(n)}(z)=\tfrac{d^{n}}{d z^{n}}F(z)$ denotes the $n$-th derivative $(n\ge0)$.
\end{lemma}

\begin{proposition} \label{prop:UV23}
The following congruence formulas are valid.
\begin{align}
U(t) &\equiv (1+120 t +96 t^{2} +128 t^{3} ) U(t^2) \pmod{2^{8}}, \label{eq:U28} \\
V(t) &\equiv (1+72 t +128 t^{2} +128 t^{3} +64 t^{4} + 128 t^{8} ) U(t^2) \pmod{2^{8}}, \label{eq:V28} \\
\begin{split}
U(t) &\equiv ( 1+120 t+54 t^{2}+189 t^{3}+135 t^{4} +81 t^{5} +162 t^{6} +81 t^{7} + 162 t^{10}) \\
&{} \quad \times U(t^{3})   \pmod{3^{5}}, 
\end{split} \label{eq:U35} \\
\begin{split}
V(t) &\equiv ( 1+111 t+216 t^{2}+162 t^{3}+135 t^{4} + 81 t^{5} +81 t^{7} +162 t^{9} +162 t^{10} )\\
&{} \quad \times U(t^{3})  \pmod{3^{5}} .
\end{split} \label{eq:V35}
\end{align}
\end{proposition}

\begin{proof}
By putting $t=z(1-432z)$ and using the quadratic transformation \eqref{eq:quadratictransf}, we have
\begin{align*}
U(t) = {}_{2}F_{1} \left( \frac{1}{12}, \frac{5}{12} ; 1 ; 1728t \right)^{2} = {}_{2}F_{1} \left( \frac{1}{6}, \frac{5}{6} ; 1 ; 432z \right)^{2} = F(z)^{2} .
\end{align*}
and then
\begin{align*}
F(z^{2})^{2} &= {}_{2}F_{1} \left( \frac{1}{6}, \frac{5}{6} ; 1 ; 432z^{2} \right)^{2} = {}_{2}F_{1} \left( \frac{1}{12}, \frac{5}{12} ; 1 ; 1728z^{2}(1-432z^{2}) \right)^{2} \\
&= {}_{3}F_{2} \left( \frac{1}{6}, \frac{1}{2}, \frac{5}{6} ; 1, 1 ; 1728z^{2}(1-432z^{2})  \right) = \sum_{r=0}^{\infty} u_{r} z^{2r} (1-432z^{2})^{r} \\
&\equiv  \sum_{r=0}^{\infty} u_{r} z^{2r} (1+80r z^{2}) \pmod{2^{8}}, \\
U(t^{2}) &\equiv {}_{3}F_{2} \left( \frac{1}{6}, \frac{1}{2}, \frac{5}{6} ; 1, 1 ; 1728z^{2}(1+160z)  \right) \equiv \sum_{r=0}^{\infty} u_{r} z^{2r} (1+160 z )^{r} \\
&\equiv  \sum_{r=0}^{\infty} u_{r} z^{2r} (1+160r z) \pmod{2^{8}} .
\end{align*}
To explicitly calculate the difference between these formal power series, we focus on the following congruence formula. By changing the variable $z^{2}=x$, we have
\begin{align*}
\frac{2z \tfrac{d}{d z} F(z^{2})}{F(z^{2})} &= 4x \frac{\tfrac{d}{d x} F(x)}{F(x)} \equiv 4x \frac{\tfrac{d}{d x} F_{8}(x)}{F_{8}(x)} \equiv 240 x+224 x^{2}+192 x^{4}+128 x^{8} \\
&\equiv 240 z^{2}+224 z^{4}+192 z^{8}+128 z^{16} \pmod{2^{8}}
\end{align*}
and then
\begin{align*}
&{} z \frac{d}{dz}F(z^{2})^{2} = F(z^{2})^{2} \cdot \frac{2z \tfrac{d}{d z} F(z^{2})}{F(z^{2})} \\
&\equiv (240 z^{2}+224 z^{4}+192 z^{8}+128 z^{16})  F(z^{2})^{2} \pmod{2^{8}}.
\end{align*}
On the other hand, from the power series expansion of $F(z^{2})^{2}$ calculated above,
\begin{align*}
z \frac{d}{dz}F(z^{2})^{2} &\equiv z \frac{d}{dz} \sum_{r=0}^{\infty} u_{r} z^{2r} (1+80r z^{2})  \\
&\equiv \sum_{r=0}^{\infty} 2r u_{r} z^{2r} + 160 \sum_{r=0}^{\infty} r(r+1) u_{r} z^{2r+2} \pmod{2^{8}}.
\end{align*}
Since we know that $u_{r} \equiv 0 \pmod{2^{3}}$ as mentioned in Remark \ref{rem:ur0mod8}, $160u_{r} \equiv 0 \pmod{2^{8}}$ holds for $r\ge0$, that is, the  second term of the above equation is congruent to $0$ modulo $2^{8}$. Therefore,
\begin{align*}
F(z^{2})^{2} - U(t^{2}) &\equiv 40z(z-2) \sum_{r=0}^{\infty} 2r u_{r} z^{r} \\
&\equiv 40z(z-2) \cdot z \frac{d}{dz}F(z^{2})^{2}
\equiv 128z^{4} F(z^{2})^{2} \pmod{2^{8}}.
\end{align*}
Hence we have $(1-128z^{4}) F(z^{2})^{2} \equiv U(t^{2}) \pmod{2^{8}}$ and so  $F(z^{2})^{2} \equiv (1+128z^{4}) U(t^{2}) \pmod{2^{8}}$. 
By combining these congruences, we obtain the congruence formula \eqref{eq:U28} for $U(t)$:
\begin{align*}
\frac{U(t)}{U(t^{2})} &= \frac{F(z^{2})^{2}}{U(t^{2})} \cdot \left( \frac{F(z)}{F(z^{2})} \right)^{2} \equiv (1+128z^{4}) \left( \frac{F_{8}(z)}{F_{7}(z^{2})} \right)^{2} \\
&\equiv 1+120z+224z^{2}+128z^{3} \equiv 1+120t+96t^{2}+128t^{3} \pmod{2^{8}} .
\end{align*}
Note that the last equality uses the congruence relation $z \equiv t(1+176t) \pmod{2^{8}}$. 
To obtain the formula \eqref{eq:V28}, we calculate the change of the variable $t=z(1-432z)$ with the following:
\begin{align*}
\frac{\tfrac{d}{d t} U(t)}{U(t)} &\equiv \frac{(1+96z) \tfrac{d}{dz} F(z)^{2}}{F(z)^{2}} \equiv 2(1+96z) \frac{\tfrac{d}{dz} F(z)}{F(z)} \equiv 2(1+96z) \frac{\tfrac{d}{dz} F_{8}(z)}{F_{8}(z)} \\
&\equiv 120+112 z+128 z^{2}+224 z^{3}+192 z^{7}+128 z^{15} \\
&\equiv 120 + 112 t + 128 t^{2} + 224 t^{3} + 192 t^{7} + 128 t^{15} \pmod{2^{8}}.
\end{align*}
Hence we have
\begin{align*}
\frac{V(t)}{U(t^{2})} &= \frac{1}{U(t^{2})}\left( 1+6 t \frac{d}{d t} \right) U(t) = \left( 1 +6 t \frac{\frac{d}{d t} U(t)}{U(t)} \right) \frac{U(t)}{U(t^{2})} \\
&\equiv 1+72 t+128 t^{2}+128 t^{3}+64 t^{4}+128 t^{8} \pmod{2^{8}}.
\end{align*}
Since equations \eqref{eq:U35} and \eqref{eq:V35} can be obtained by the similar calculation with $p=3$ and $s=4$ in Lemma \ref{lem:Dwork23}, we omit the proof. 
\end{proof}

From Propositions \ref{prop:UV57} and \ref{prop:UV23}, we obtain the following corollary.
\begin{corollary} The following formal infinite product expressions hold.
\begin{align*}
U(t) &\equiv \prod_{k=0}^{\infty} (1+120t^{2^{k}} +96t^{2^{k+1}} +128t^{3\cdot 2^{k}}) \pmod{2^{8}} , \\
U(t) &\equiv \prod_{k=0}^{\infty}  \genfrac{(}{)}{0pt}{}{1+120 t^{3^{k}}+54 t^{2\cdot 3^{k}}+189 t^{3^{k+1}}+135 t^{4 \cdot 3^{k}} }{+81 t^{5 \cdot 3^{k}} +162 t^{6 \cdot 3^{k}} +81 t^{7 \cdot 3^{k} } + 162 t^{10 \cdot 3^{k}} } \pmod{3^{5}}, \\
U(t) &\equiv \prod_{k=0}^{\infty} (1+20t^{5^{k}} +10t^{2\cdot5^{k}})  \pmod{5^{2}}, \\
U(t) &\equiv \prod_{k=0}^{\infty} \frac{(1+22t^{7^{k}} +7t^{2\cdot 7^{k}} +21t^{3 \cdot 7^{k}} + t^{7^{k+1}} +36t^{8\cdot 7^{k}} )}{(1+t^{7^{k+1}})} \pmod{7^{2}} .
\end{align*}
\end{corollary}

\begin{question}
Recall that $U(j(\tau)^{-1}) = E_{4}(\tau)^{1/2}$ and $V(j(\tau)^{-1}) = E_{2}(\tau) E_{4}(\tau)^{3/2} E_{6}(\tau)^{-1}$. What are the counterparts of the above congruence formulas in the (quasi) modular form side? 
\end{question}

We still cannot answer this vague question, but we can prove the following assertion as a closely related result. (Recall that $U(t)= \mathcal{F}_{1}(t)^{2}$.)
\begin{proposition}\label{prop:infiniteproductF1}
Let $\mu(n)$ be the M\"{o}bius function and $(\cdot,\cdot)$ be the Atkin inner product defined as
\begin{align*}
(f, g) = \text{constant term of $f g E_{2}$ as a Laurent series in $q$.} \quad ( f(\tau),g(\tau) \in \mathbb{C}[j(\tau)] ).
\end{align*}
Then the formal power series $\mathcal{F}_{1}(t)$ can be expressed formally as follows:
\begin{align}
\mathcal{F}_{1}(t) &= {}_{2}F_{1} \left( \frac{1}{12}, \frac{5}{12} ; 1 ; 1728 t   \right) = \prod_{n=1}^{\infty} (1-t^{n})^{-c(n)} , \quad c(n) = \frac{1}{12n} \sum_{d|n} \mu (\tfrac{n}{d})\, (j^{d},1) , \label{eq:F1infiniteprod} \\
\mathcal{F}_{1}(t) &= \exp \left( \sum_{m=1}^{\infty} \frac{1}{12}(j^{m},1) \, \frac{t^{m}}{m} \right). \label{eq:F1exp}
\end{align}
\end{proposition}
Here are some examples of exponents $c(n)$.
\[ \{c(n) \}_{n\ge1} = \{60,37950,36139180,40792523310,50608476466548, \cdots \}. \]

\begin{proof}
As already mentioned in \cite[\S 5]{kaneko1998supersingular}, the moment-generating function of the Atkin inner product is given by
\begin{align}
\frac{E_{2}(\tau) E_{4}(\tau)}{j(\tau) E_{6}(\tau)} = \sum_{m=0}^{\infty} \frac{(j^{m},1)}{j(\tau)^{m+1}} = \frac{1}{j(\tau)} + \frac{720}{j(\tau)^{2}} + \frac{911520}{j(\tau)^{3}} + \frac{1301011200}{j(\tau)^{4}} + \cdots . \label{eq:moment-genfunc}
\end{align}
By transforming the left-hand side of the above equation using the hypergeometric expressions of the Eisenstein series in Proposition \ref{prop:EisensteinHyp}, we have
\begin{align*}
\sum_{m=0}^{\infty} (j^{m},1) \, t^{m} &=  (1-1728t)^{-1/2} \cdot \frac{{}_{2}F_{1} \left( -\frac{1}{12}, \frac{7}{12} ; 1 ; 1728 t   \right)}{{}_{2}F_{1} \left( \frac{1}{12}, \frac{5}{12} ; 1 ; 1728 t  \right)} = \frac{{}_{2}F_{1} \left( \frac{13}{12}, \frac{5}{12} ; 1 ; 1728 t   \right)}{{}_{2}F_{1} \left( \frac{1}{12}, \frac{5}{12} ; 1 ; 1728 t   \right)}  \\
&= \left\{ \left( 1+12t\frac{d}{d t} \right) {}_{2}F_{1} \left( \frac{1}{12}, \frac{5}{12} ; 1 ; 1728 t  \right) \right\} {}_{2}F_{1} \left( \frac{1}{12}, \frac{5}{12} ; 1 ; 1728 t   \right)^{-1} \\
&= 1 + 12 t \frac{d}{d t} \log {}_{2}F_{1} \left( \frac{1}{12}, \frac{5}{12} ; 1 ; 1728 t   \right) \\
&= 1 + 12 t \frac{d}{d t} \log \prod_{n=1}^{\infty} (1-t^{n})^{-c(n)} \\
&= 1 +12 \sum_{n=1}^{\infty} c(n) \frac{n t^{n}}{1-t^{n}} .
\end{align*}
Therefore, we have $(j^{m},1)= 12\sum_{d|m}d \cdot c(d)$ by comparing the coefficients of $t^{m}$ on both sides of the above equation, and obtain the desired expression \eqref{eq:F1infiniteprod} of $c(n)$ by using the M\"{o}bius inversion formula. Equation \eqref{eq:F1exp} is obtained by dividing both sides of the equation
\begin{align*}
t \frac{d}{d t} \log {}_{2}F_{1} \left( \frac{1}{12}, \frac{5}{12} ; 1 ; 1728 t   \right) = \sum_{m=1}^{\infty} \frac{1}{12}(j^{m},1)\, t^{m}
\end{align*}
by $t$ and then integrating with respect to $t$.
\end{proof}

\begin{remark}
Equation \eqref{eq:moment-genfunc}, and therefore equation \eqref{eq:F1exp}, is equivalent to:
\begin{align*}
\Delta (\tau ) = j(\tau)^{-1} \exp \left( \sum_{m=1}^{\infty} (j^{m},1) \, \frac{j(\tau )^{-m}}{m} \right). 
\end{align*}
\end{remark}

%%%%%%%%%%%%%%%%%%%%%%%%%%%%%%%%%%%%%%%%%%%%%%%%%%%%%%%%%%%%%%%%%%%%%%%%%%%%%%%%%%%%%%%%%
%%%%%%%%%%%%%%%%%%%%%%%%%%%%%%%%%%%%%%%%%%%%%%%%%%%%%%%%%%%%%%%%%%%%%%%%%%%%%%%%%%%%%%%%%
%%%%%%%%%%%%%%%%%%%%%%%%%%%%%%%%%%%%%%%%%%%%%%%%%%%%%%%%%%%%%%%%%%%%%%%%%%%%%%%%%%%%%%%%%

\section{Proof of the main theorem} \label{sec:prfmainthm}
First, we summarize the normalizing factors $N_{m,a}$ and the polynomials $A_{m,a}(X)$ and $B_{m,a}(X)$ that appear in the main theorem. 
From the definition of the normalizing factor in Section \ref{sec:intro}, we see  that its prime factors do not exceed the weight of the corresponding normalized extremal quasimodular form.

%%%%%%%%%%%%%%%%%%%%%%%%%%%%%%%%%%%%%%%%%
\textbf{$\bullet$ Case of weight $w \in \mathcal{S}_{0} = \{ 12,24 \}$.} \\
Prime factorization of normalizing factors.
\begin{align*}
N_{1,0} &= 2^5\cdot 3^3\cdot 5\cdot 7\cdot 11 , \\
N_{2,0} &= 2^6\cdot 3^3\cdot 5\cdot 7\cdot 11\cdot 13\cdot 17\cdot 19\cdot 23 .
\end{align*}
Atkin-like polynomials.
\begin{align*}
A_{1,0}(X) = 1 , \; A_{2,0}(X) = X-824 .
\end{align*}
Adjoint polynomials.
\begin{align*}
B_{1,0}(X) = X-1008, \; B_{2,0}(X) = X^2-1832 X+497952 .
\end{align*}

%%%%%%%%%%%%%%%%%%%%%%%%%%%%%%%%%%%%%%%%%
\textbf{$\bullet$ Case of weight $w \in \mathcal{S}_{2} = \{ 2,14,38 \}$.} \\
Prime factorization of normalizing factors.
\begin{align*}
N_{0,2} &= 1 , \\
N_{1,2} &= 2^5\cdot 3^3\cdot 5\cdot 7\cdot 13 , \\
5N_{3,2} &= 2^7\cdot 3^4\cdot 5^2\cdot 7\cdot 11\cdot 13\cdot 17\cdot 19\cdot 23\cdot 29\cdot 31\cdot 37  .
\end{align*}
(Original) Atkin polynomials.
\begin{align*}
A_{0,2}(X) &= 1, \; A_{1,2}(X) = X-720, \\
A_{3,2}(X) &= X^3-\frac{12576 
}{5} X^2 +1526958 X-107765856 .
\end{align*}
Adjoint polynomials.
\begin{align*}
B_{0,2}(X) = 0, \; B_{1,2}(X) = 1, \; B_{3,2}(X) = X^2-\frac{8976}{5} X +627534 .
\end{align*}

%%%%%%%%%%%%%%%%%%%%%%%%%%%%%%%%%%%%%%%%%
\textbf{$\bullet$ Case of weight $w \in \mathcal{S}_{6} = \{6,18,30,54,114\}$.} \\
Prime factorization of normalizing factors.
\begin{align*}
N_{0,6} &= 2^4\cdot 3^2\cdot 5 , \\
N_{1,6} &= 2^6\cdot 3^3\cdot 5\cdot 7\cdot 11\cdot 13\cdot 17 , \\
N_{2,6} &= 2^6\cdot 3^4\cdot 5^2\cdot 7\cdot 11\cdot 13\cdot 17\cdot 19\cdot 23\cdot 29 , \\
2 N_{4,6} &= 2^7\cdot 3^4\cdot 5^2\cdot 7^2\cdot 11\cdot 13\cdot 17\cdot 19\cdot 23\cdot 29\cdot 31\cdot 37\cdot 41\cdot 43\cdot 47\cdot 53 , \\
51N_{9,6} &= 2^8\cdot 3^5\cdot 5^2\cdot 7^2\cdot 11\cdot 13\cdot 17\cdot 19\cdot 23\cdot 29\cdot 31\cdot 37\cdot 41\cdot 43\cdot 47\cdot 53\cdot 59 \\
&{} \quad \cdot 61\cdot 67\cdot 71\cdot 73\cdot 79\cdot 83\cdot 89\cdot 97\cdot 101\cdot 103\cdot 107\cdot 109\cdot 113 .
\end{align*}
Atkin-like polynomials.
\begin{align*}
A_{0,6}(X) &= 1, \quad A_{1,6}(X) = X-1266 , \quad A_{2,6}(X) = X^2-2115 X+870630 , \\
A_{4,6}(X) &= X^4-\frac{7671}{2} X^3 +4871313 X^2-2260803660 X+273189722310, \\
A_{9,6}(X) &= X^9-\frac{24454}{3} X^{8}+\frac{474979296}{17} X^{7} -\frac{888804457205}{17} X^6  \\
&{} \quad +58002865348421 X^5 -38759471954111394 X^4 \\
&{} \quad +15135088185868167792 X^3-3173598010686486090312 X^2 \\
&{} \quad +297473555337690122052390 X-7840346480159903987708940 .
\end{align*}
Adjoint polynomials.
\begin{align*}
B_{0,6}(X) &= 1, \quad B_{1,6}(X) = X-546, \quad B_{2,6}(X) = X^2-1395 X+259350 , \\
B_{4,6}(X) &= X^4-\frac{6231}{2} X^3 +3021273 X^2-948582060 X+53723885670 , \\
B_{9,6}(X) &= X^9-\frac{22294}{3} X^8 +\frac{390702816}{17}  X^7 -\frac{651013930805 }{17} X^6  \\
&{} \quad +37180279576181 X^5 -21228003877921074 X^4 \\
&{} \quad +6835398004395374832 X^3-1114698418843177975752 X^2 \\
&{} \quad +72322444486635699257190 X-919318930586739576036780 .
\end{align*}

%%%%%%%%%%%%%%%%%%%%%%%%%%%%%%%%%%%%%%%%%
\textbf{$\bullet$ Case of weight $w \in \mathcal{S}_{8} = \{8,20,32,68,80\}$.} \\
Prime factorization of normalizing factors.
\begin{align*}
N_{0,8} &= 2^4\cdot 3^2\cdot 7 , \\
N_{1,8} &= 2^6\cdot 3^3\cdot 5\cdot 7\cdot 11\cdot 13\cdot 19 , \\
N_{2,8} &= 2^6\cdot 3^4\cdot 5^2\cdot 7\cdot 11\cdot 13\cdot 17\cdot 19\cdot 23\cdot 31 , \\
5N_{5,8} &= 2^8\cdot 3^4\cdot 5^2\cdot 7^2\cdot 11\cdot 13\cdot 17\cdot 19\cdot 23\cdot 29\cdot 31\cdot 37\cdot 41\cdot 43\cdot 47 \cdot 53\cdot 59 \\
&{} \quad \cdot 61\cdot 67 , \\
11N_{6,8} &= 2^8\cdot 3^4\cdot 5^2\cdot 7^2\cdot 11\cdot 13\cdot 17\cdot 19\cdot 23\cdot 29\cdot 31\cdot 37\cdot 41\cdot 43\cdot 47\cdot 53\cdot 59 \\
&{} \quad \cdot 61\cdot 67\cdot 71\cdot 73\cdot 79 .
\end{align*}
Atkin-like polynomials.
\begin{align*}
A_{0,8}(X) &= 1, \quad A_{1,8}(X) = X-330, \quad A_{2,8}(X) = X^2-1215 X+129030, \\
A_{5,8}(X) &= X^5-\frac{19098}{5} X^4 +\frac{25015408}{5} X^3 -\frac{12959037322 }{5} X^2 +441761976414 X \\
&{} \quad -9018997829292  , \\
A_{6,8}(X) &= X^6-4685 X^5+\frac{89349390}{11}  X^4 -6372443376 X^3+2195718854056 X^2 \\
&{} \quad -261120476348550 X+3783879543834780 .
\end{align*}
Adjoint polynomials.
\begin{align*}
B_{0,8}(X) &= 1, \quad B_{1,8}(X) = X-1338, \quad B_{2,8}(X) = X^2-2223 X+1021110 , \\
B_{5,8}(X) &= X^5-\frac{24138}{5}  X^4 +\frac{42602992}{5} X^3 -\frac{33192286666 }{5} X^2 +\frac{10734540754806}{5} X \\
&{} \quad -202399435400844 , \\
B_{6,8}(X) &= X^6-5693 X^5+\frac{137637630}{11}  X^4 -\frac{146033508816}{11}  X^3 +6911247661864 X^2 \\
&{} \quad -1568906774156358 X+105994437115386300 .
\end{align*}

\begin{proof}[Proof of Theorem \ref{thm:mainthm}]
We prove only the non-trivial and highest weight case $w=114 \in \mathcal{S}_{6}$. 
The remaining cases can be proved in a similar way. 

Since $G_{114}^{(1)}(\tau) \in \mathbb{Z}[\![q]\!] \Leftrightarrow P_{19}(t) \in \mathbb{Z}[\![t]\!]$ holds from \eqref{claim:GPQ} of Lemma \ref{lem:integrality}, we transform the formal power series $P_{19}(t)$ according to \eqref{eq:P2m1AVBU} as follows:
\begin{align*}
N_{9,6}\, t^{19}(1-1728t)^{-1/2} P_{19}(t) = \widetilde{A_{9,6}}(t) V(t) - \widetilde{B_{9,6}}(t) U(t).
\end{align*}
Since the polynomials $A_{9,6}(X)$ and $B_{9,6}(X)$ belong to $\frac{1}{51}\mathbb{Z}[X]$, we multiply the both sides of the above equation by $51$, so that the right-hand side is the power series with integral coefficients. 
Using Proposition \ref{prop:UVLucas}, we obtain the following congruence formula for prime numbers $p\ge11$: 
\begin{align*}
&{} 51 \left( \widetilde{A_{9,6}}(t) V(t) - \widetilde{B_{9,6}}(t) U(t) \right) \\
&\equiv \left\{ 51 \widetilde{A_{9,6}}(t)  \sum_{m=0}^{[p/6]} (6m+1) u_{m} t^{m} -  51 \widetilde{B_{9,6}}(t)  \sum_{m=0}^{[p/6]}u_{m} t^{m} \right\} U(t^{p})  \pmod{p}.
\end{align*}
By direct calculation using \texttt{Mathematica}, we can see that the polynomial part of the right-hand side of the above equation is congruent to 0 modulo $p \; (11\le p \le 113)$. 
To perform similar calculations for modulo $2^{8}, 3^{5}, 5^{2}$, and $7^{2}$, we use Propositions \ref{prop:UV57} and \ref{prop:UV23}. For example, from \eqref{eq:U28} and \eqref{eq:V28}, we have
\begin{align*}
&{} 51 \left( \widetilde{A_{9,6}}(t) V(t) - \widetilde{B_{9,6}}(t) U(t) \right) \\
&\equiv  \left\{ 51 \widetilde{A_{9,6}}(t) (1+72 t +128 t^{2} +128 t^{3} +64 t^{4} + 128 t^{8} ) \right. \\
&{} \quad \left. - 51 \widetilde{B_{9,6}}(t) (1+120 t +96 t^{2} +128 t^{3} ) \right\} U(t^2) \pmod{2^8}
\end{align*}
and then the polynomial part of the right-hand side is congruent to 0 modulo $2^{8}$. 
The same assertion holds for the remaining cases. 
Combining the results for each of these primes, we have the desired congruence 
\begin{align*}
&{} 51 N_{9,6}\, t^{19}(1-1728t)^{-1/2} P_{19}(t) \\
&= 51 \left( \widetilde{A_{9,6}}(t) V(t) - \widetilde{B_{9,6}}(t) U(t) \right) \equiv 0 \pmod{51 N_{9,6}}
\end{align*}
and hence $t^{19} (1-1728t)^{-1/2} P_{19}(t) \in \mathbb{Z}[\![t]\!]$, and so $P_{19}(t) \in \mathbb{Z}[\![t]\!]$ holds from \eqref{claim:f12} in Lemma \ref{lem:integrality}.
\end{proof}

\begin{question}
The proof of Theorem \ref{thm:mainthm} in this paper is a ``hypergeometric" proof. Can this theorem be proved using only the theory of modular forms?
\end{question}

%%%%%%%%%%%%%%%%%%%%%%%%%%%%%%%%%%%%%%%%%%%%%%%%%%%%%%%%%%%%%%%%%%%%%%%%%%%%%%%%%%%%%%%%%
%%%%%%%%%%%%%%%%%%%%%%%%%%%%%%%%%%%%%%%%%%%%%%%%%%%%%%%%%%%%%%%%%%%%%%%%%%%%%%%%%%%%%%%%%
%%%%%%%%%%%%%%%%%%%%%%%%%%%%%%%%%%%%%%%%%%%%%%%%%%%%%%%%%%%%%%%%%%%%%%%%%%%%%%%%%%%%%%%%%

\section{Other choice of the leading coefficients} \label{sec:moreint}
Why did we choose the leading coefficient as 1 in Definition \ref{def:exqmf} of the normalized extremal quasimodular form? Was it really a natural choice? Of course, the set of weights $w$ of $c\, G_{w}^{(r)}$ with integral Fourier coefficients changes if we choose a number $c$ other than 1 as the leading coefficient. 
In this section we show that $c\, G_{w}^{(1)}$ has integral Fourier coefficients if we choose a constant~$c$ based on the hypergeometric expressions of $G_{w}^{(1)}$ in Proposition \ref{prop:Gto2F1}.

The reason we often focus on (positive) integers is that we can expect them to count ``something'', such as the dimension or degree or order of some mathematical objects, the number of curves or points with certain arithmetic or geometric properties, and so on.
Interestingly, Nebe proved the integrality of the Fourier coefficients of $G_{14}^{(1)}$ by using the properties of the automorphism group of the Leech lattice $\Lambda_{24}$ in the appendix of \cite{pellarin2020extremal}\footnote{The author speculates that Nebe probably uses $q$ to mean $e^{\pi i \tau}$ in the proof of Theorem A.1. Hence the correct expression of $f_{1,14}$ for $q=e^{2 \pi i \tau}$ is given by $A^{-1} \sum_{a=1}^{\infty} \tfrac{a}{2} | L_{a} | q^{a/2}$.}:
\begin{align*}
\theta_{\Lambda_{24}} &= \sum_{\lambda \in \Lambda_{24}} e^{\pi i \tau \| \lambda \|^{2}} = \sum_{n=0}^{\infty} |\{ \lambda \in \Lambda_{24} \mid \|\lambda\|^{2}=n  \}| \,q^{n/2} \\
&= E_{4}^{3} -720 \Delta = E_{12} - \frac{65520}{691} \Delta  \\
&= 1+196560 q^{2}+16773120 q^{3}+398034000 q^{4}+4629381120 q^{5} +O(q^{6})    , \\
G_{14}^{(1)} &= \frac{E_{2} (E_{4}^{3} - 720 \Delta)  - E_{4}^{2}E_{6}}{393120} = \frac{D(\theta_{\Lambda_{24}})}{393120} \\
&= q^{2}+128 q^{3}+4050 q^{4}+58880 q^{5}+525300 q^{6} +O(q^{7}) .
\end{align*}
Therefore, the number 393120 is twice the number of lattice vectors of squared norm 4 in $\Lambda_{24}$, and $N_{1,2}$ is exactly equal to this number. 
It is known that the Eisenstein series $E_{4}$ is the theta series of the $E_{8}$-lattice. Since $G_{6}^{(1)} = D(E_{4})/240 = D(\theta_{E_{8}})/240$, as with $G_{14}^{(1)}$, the Fourier coefficients of $G_{6}^{(1)}$ are related to the number of lattice vectors in the $E_{8}$-lattice. 
Note that although $G_{6}^{(1)}$ and $G_{14}^{(1)}$ were obtained as derivatives of a certain modular form, the only such $G_{w}^{(1)}$ that can be obtained this way are for $w=6,8,10,14$. This fact is easily seen by comparing the dimension formulas of $M_{w}$ and $QM_{w+2}^{(1)}$. 

Inspired by these coincidences\footnote{The modular solutions of the Kaneko--Zagier equation \eqref{eq:2ndKZeqn} with small weights are also closely related to the theta series of the ADE-type root lattice. For more details, see \cite[p.~158]{kaneko2003modular}.} $\tfrac{1}{3}N_{0,6}\, G_{6}^{(1)}=D(\theta_{E_{8}}) \in \mathbb{Z}[\![q]\!]$ and $N_{1,2}\, G_{14}^{(1)}=D(\theta_{\Lambda_{24}}) \in \mathbb{Z}[\![q]\!]$, although we still do not know what the Fourier coefficients of these forms counting up, we have arrived at the following theorem.

\begin{theorem} \label{thm:intNG}
Let $N_{m,a}$ be the normalizing factor defined by \eqref{eq:d1nf02} and \eqref{eq:d1nf68}. Then, for any $m\ge0$, the Fourier coefficients of the following extremal quasimodular forms are all integers: 
\begin{align*}
\frac{N_{m,0}}{24m} \,G_{12m}^{(1)}, \; \frac{N_{m,2}}{24m} \,G_{12m+2}^{(1)}, \; \frac{N_{m,6}}{12(2m+1)} \,G_{12m+6}^{(1)} , \; \frac{N_{m,8}}{12(2m+1)} \,G_{12m+8}^{(1)} . 
\end{align*}
Furthermore, since $G_{6n+4}^{(1)}= E_{4} G_{6n}^{(1)}$, the forms $\tfrac{1}{24m} N_{m,0} \, G_{12m+4}^{(1)}$ and $\tfrac{1}{12(2m+1)} N_{m,6} \, G_{12m+10}^{(1)}$ have the same properties. (Note: Since the factors $N_{m,0}$ and $N_{m,2}$ are defined by the product of binomial coefficients \eqref{eq:d1nf02}, we substitute $m=0$ after dividing by $m$ and reducing.)
\end{theorem}
In the proof of this theorem, we use the following power series instead of $q$ as the local parameter at infinity, based on the method of Remark \ref{rem:E4fourthroot}:
\begin{align}
&{} \frac{z(q)}{432} = \frac{1}{864} \left\{ 1 - \left( 1- \frac{1728}{j(q)} \right)^{1/2}  \right\} = \frac{1}{864} \left( 1 - E_{4}(q)^{-3/2} E_{6}(q) \right) \label{eq:z432} \\
&=q-312 q^{2}+87084 q^{3}-23067968 q^{4}+5930898126 q^{5} +O(q^{6}).  \notag
\end{align}

\begin{proof}
Throughout this proof, we consider $z$ to be $z=\tfrac{1}{2}\{ 1- (1-1728t)^{1/2} \} \in 432t\mathbb{Z}[\![t]\!]$ and $t=j^{-1}$.
From Proposition \ref{prop:Gto2F1}, the normalized extremal quasimodular forms $G_{w}^{(1)}$ are the product of the power of $\mathcal{F}_{1}(t)={}_{2}F_{1}\left( \frac{1}{12}, \frac{5}{12} ; 1 ; 1728 t \right)$  and a hypergeometric series. Furthermore, since $\mathcal{F}_{1}(t) \in 1 +t \mathbb{Z}[\![t]\!]$ holds from Remark \ref{rem:E4fourthroot}, we have $\mathcal{F}_{1}(t)^{w-1} \in 1 +t \mathbb{Z}[\![t]\!]$. Therefore, as already mentioned in Lemma \ref{lem:integrality} \eqref{claim:fqj}, it suffices to show that $C_{w} \mathcal{F}_{1}(t)^{1-w} G_{w}^{(1)} \in \mathbb{Z}[\![t]\!]$ to prove the theorem, where the constant $C_{w}$ is an appropriate normalization constant.

First we calculate the case of $w=12m$. From the quadratic transformation formula  \eqref{eq:quadratictransf} we have
\begin{align}
&{} \frac{N_{m,0}}{24m}\,  {}_{2}F_{1} \left( m + \frac{1}{12}, m + \frac{5}{12} ; 2m+1 ; 1728 t  \right)  \notag   \\  
&= \frac{(12m)!}{(2m)!(4m)!(6m)!}\, {}_{2}F_{1} \left( 2m + \frac{1}{6}, 2m + \frac{5}{6} ; 2m+1 ; z  \right) \notag \\
&=  \frac{(12m)!}{(2m)!(4m)!(6m)!} \cdot \frac{(1)_{2m}}{(\tfrac{1}{6})_{2m} (\tfrac{5}{6})_{2m}} \cdot \frac{d^{2m}}{d z^{2m}} {}_{2}F_{1} \left( \frac{1}{6}, \frac{5}{6} ; 1 ; z  \right) \notag   \\
&= \sum_{n=0}^{\infty} \binom{2m+n}{n} \binom{6m+3n}{2m+n} \binom{12m+6n}{6m+3n} \left( \frac{z}{432} \right)^{n}  \in \mathbb{Z}[\![t]\!] .   \label{eq:binom12m}
\end{align}
It is convenient to use the following formulas to calculate the last equality:
\begin{align*}
432^{k} \left(\frac{1}{6}\right)_{k} \left(\frac{5}{6}\right)_{k} = \frac{\Gamma(k)\Gamma(6k)}{\Gamma(2k)\Gamma(3k)} = \frac{\Gamma(k+1)\Gamma(6k+1)}{\Gamma(2k+1)\Gamma(3k+1)}, \quad (\alpha + k)_{\ell} = \frac{(\alpha)_{k+\ell}}{(\alpha)_{k}}.
\end{align*}
A similar calculation for $w=12m+6$ yields
\begin{align}
&{} \frac{N_{m,6}}{12(2m+1)} \,  {}_{2}F_{1} \left( m + \frac{7}{12}, m + \frac{11}{12} ; 2m+2 ; 1728 t  \right) \notag   \\
&= \sum_{n=0}^{\infty} \binom{2m+n+1}{n} \binom{6m+3n+3}{2m+n+1} \binom{12m+6n+6}{6m+3n+3} \left( \frac{z}{432} \right)^{n} \in \mathbb{Z}[\![t]\!] .  \label{eq:binom12m6}
\end{align}
To prove the assertions about $G_{12m+2}^{(1)}$ and $G_{12m+8}^{(1)}$, we recall $\tfrac{1}{1-2n} \tbinom{2n}{n} \in \mathbb{Z}$. Then, we can see that the coefficient of $(z/432)^{n}$ of the following formal power series is an integer: 
\begin{align*}
{}_{2}F_{1} \left( -\frac{1}{6}, \frac{7}{6} ; 1 ; z  \right) = \sum_{n=0}^{\infty} (6n+1)\binom{3n}{n} \cdot \frac{1}{1-6n} \binom{6n}{3n} \cdot \left( \frac{z}{432}\right)^{n}  \in \mathbb{Z}[\![t]\!] .
\end{align*}
Hence we have
\begin{align*}
&{} \frac{N_{m,2}}{24m} \, {}_{2}F_{1} \left( m - \frac{1}{12}, m + \frac{7}{12} ; 2m+1 ; 1728 t  \right) \\
&= \frac{N_{m,2}}{24m} {}_{2}F_{1} \left( 2m - \frac{1}{6}, 2m + \frac{7}{6} ; 2m+1 ; z  \right) \\
&= \frac{N_{m,2}}{24m} \cdot \frac{(1)_{2m}}{(-\tfrac{1}{6})_{2m} (\tfrac{7}{6})_{2m}} \cdot \frac{d^{2m}}{d z^{2m}} {}_{2}F_{1} \left( -\frac{1}{6}, \frac{7}{6} ; 1 ; z  \right) \\
&= \sum_{n=0}^{\infty} \frac{12m+6n+1}{12m+6n-1} \binom{2m+n}{n} \binom{6m+3n}{2m+n} \binom{12m+6n}{6m+3n} \left( \frac{z}{432} \right)^{n} \in \mathbb{Z}[\![t]\!] 
\end{align*}
and also have
\begin{align*}
&{} \frac{N_{m,8}}{12(2m+1)} \, {}_{2}F_{1} \left( m + \frac{5}{12}, m + \frac{13}{12} ; 2m+2 ; 1728 t  \right) \\
&= \sum_{n=0}^{\infty} \frac{12m+6n+7}{12m+6n+5} \binom{2m+n+1}{n} \binom{6m+3n+3}{2m+n+1} \binom{12m+6n+6}{6m+3n+3} \left( \frac{z}{432} \right)^{n} \in \mathbb{Z}[\![t]\!] .
\end{align*}
This completes the proof of Theorem \ref{thm:intNG}.
\end{proof}

Considering the divisors of the normalizing factors $N_{m,a}$, Theorem \ref{thm:intNG} implies the following claim. 
\begin{corollary} \label{cor:denomd1exqmf}
The denominators of the Fourier coefficients of the normalized extremal quasimodular forms $G_{w}^{(1)}$ are only divisible by prime numbers $<w$. 
\end{corollary}
This claim was originally stated as $G_{w}^{(r)} \in \mathbb{Z} \bigl[ \tfrac{1}{p} : p<w \bigr][\![q]\!] \; (1\le r \le 4)$ in 
\cite{kaneko2006extremal}. 
When $(w,r)=(6n,1)$ and $(w,r)=(6n+2,1), (6n+4,1)$, they are proved in \cite[Thm. 3.3]{pellarin2020extremal} and \cite[Thm. 1.8, Thm. 1.9]{mono2020conjecture}, respectively, and the original claim is proved in \cite{grabner2020quasimodular}. Note that our proof is mainly based on hypergeometric  expressions of $G_{w}^{(1)}$, which is different from their proof.

According to Theorem \ref{thm:intNG}, there exists a constant $c_{w}$ such that $c_{w} \, G_{w}^{(1)} \in \mathbb{Z}[\![q]\!]$. However, this constant is probably too large, and integrality holds for smaller numbers. For example, although $\tfrac{1}{48} N_{2,2}\, G_{26}^{(1)} =1454953500\, G_{26}^{(1)} \in \mathbb{Z}[\![q]\!]$ holds, we can show that the stronger claim $5 G_{26}^{(1)} \in \mathbb{Z}[\![q]\!]$ holds in the similar way as in the proof of Theorem \ref{thm:mainthm}. As a more general setting, we define the number $c_{w}^{(r)}$ as   $c_{w}^{(r)} := \min \{\, c \,|\, c\, G_{w}^{(r)} \in \mathbb{Z}[\![q]\!] \}$.
How do such numbers depend on weight and depth?

\begin{remark}
In the study of the denominator of Atkin polynomials \cite[\S 9]{kaneko1998supersingular}, the formal power series $\Phi_{n}(t)$ and $\Psi_{n}(t)$, which are defined as follows, play a central role: 
\begin{align*}
\Phi_{n+1}(t) = \Phi_{n}(t) - \lambda_{n}^{-}\, t\, \Phi_{n-1}(t), \quad \Psi_{n+1}(t) = \Psi_{n}(t) - \lambda_{n}^{+}\, t\, \Psi_{n-1}(t) \quad (n\ge1),
\end{align*}
where $\lambda_{1}^{-}=84, \; \lambda_{1}^{+}=-60$ and for $n>1$, 
\begin{align*}
\lambda_{n}^{\mp} = 12 \left( 6 \mp \dfrac{(-1)^{n}}{n-1} \right) \left( 6 \mp \dfrac{(-1)^{n}}{n} \right) .
\end{align*}
The initial power series are given by the following. 
\begin{align*}
\Phi_{0}(t) &= {}_{2}F_{1}\left( \tfrac{1}{12}, \tfrac{5}{12} ; 1 ; 1728 t \right)^{2}, \\
\Phi_{1}(t) &= 84 t\, {}_{2}F_{1}\left( \tfrac{1}{12}, \tfrac{5}{12} ; 1 ; 1728 t \right) {}_{2}F_{1}\left( \tfrac{5}{12}, \tfrac{13}{12} ; 2 ; 1728 t \right), \\
\Psi_{0}(t) &= {}_{2}F_{1}\left( \tfrac{1}{12}, \tfrac{5}{12} ; 1 ; 1728 t \right) {}_{2}F_{1}\left( -\tfrac{1}{12}, \tfrac{7}{12} ; 1 ; 1728 t \right), \\
\Psi_{1}(t) &= -60 t\, {}_{2}F_{1}\left( \tfrac{1}{12}, \tfrac{5}{12} ; 1 ; 1728 t \right) {}_{2}F_{1}\left( \tfrac{7}{12}, \tfrac{11}{12} ; 2 ; 1728 t \right) .
\end{align*}
Although the definition for the power series $\Psi_{n}(t)$ is not given in \cite{kaneko1998supersingular}, the details are omitted here because it can be derived from the same consideration as $\Phi_{n}(t)$. In fact, these power series and the power series $P_{n}(t)$ and $Q_{n}(t)$ defined by \eqref{eq:Pnt} and \eqref{eq:Qnt} have the following correspondences:
\begin{align*}
\Phi_{2m}(t) &= \tfrac{1}{12} N_{m,0}\, t^{2m} P_{2m}(t) , \quad \Phi_{2m+1}(t) = \tfrac{1}{12} N_{m,8}\, t^{2m+1} Q_{2m+1}(t), \\
\Psi_{2m}(t) &= - \tfrac{1}{12} N_{m,2}\, t^{2m} Q_{2m}(t) , \quad \Psi_{2m+1}(t) = - \tfrac{1}{12} N_{m,6}\, t^{2m+1} P_{2m+1}(t) .
\end{align*}
These correspondences can be proved by using a contiguous relation of a certain hypergeometric series. In \cite[p.~120]{kaneko1998supersingular}, it is stated without proof that ``the power series $\Phi_{n}(t)$ has integral coefficients and is divisible by $t^{n}$ for all $n$", which is essentially equivalent to a weaker version of our Theorem \ref{thm:intNG}. However, it should be emphasized that essentially the equivalent forms appeared in \cite{kaneko1998supersingular} before the introduction of the normalized extremal quasimodular forms in  \cite{kaneko2006extremal}. 
\end{remark}

%%%%%%%%%%%%%%%%%%%%%%%%%%%%%%%%%%%%%%%%%%%%%%%%%%%%%%%%%%%%%%%%%%%%%%%%%%%%%%%%%%%%%%%%%
%%%%%%%%%%%%%%%%%%%%%%%%%%%%%%%%%%%%%%%%%%%%%%%%%%%%%%%%%%%%%%%%%%%%%%%%%%%%%%%%%%%%%%%%%
%%%%%%%%%%%%%%%%%%%%%%%%%%%%%%%%%%%%%%%%%%%%%%%%%%%%%%%%%%%%%%%%%%%%%%%%%%%%%%%%%%%%%%%%%

\section{Further directions} \label{sec:furtherdirections}
Up to the previous section we have discussed the extremal quasimodular forms of depth $1$ on $\Gamma=SL_{2}(\mathbb{Z})$. It is therefore natural to consider a generalization that changes these factors. 
The case of depth $2$ on $\Gamma$ has already been considered in \cite{kaneko2006extremal,grabner2020quasimodular,kaminaka2021extremal}, but we note that in these papers there is no perspective of using a generalized hypergeometric series.
Therefore, we focus mainly on these hypergeometric aspects and present in this section some complementary results of previous studies without detailed proofs.

To describe the case of depth $\ge2$, we introduce the Rankin--Cohen brackets (\cite[\S 5.2]{zagier2008elliptic}, see also \cite{zagier1994modularforms,cohen1997automorphic}), which is defined for integers $k,\ell , n \ge 0$ and functions $f,g$ on $\mathfrak{H}$ by
\begin{align*}
[f,g]_{n}^{(k,\ell)} := \sum_{i=0}^{n} (-1)^{i} \binom{n+k-1}{n-i} \binom{n+\ell -1}{i} D^{i}(f) D^{n-i}(g) .
\end{align*}
We then define the $(r+1)$-th order differential operator $\theta_{k}^{(r)}$ as follows.
\begin{align*}
\theta_{k}^{(r)}(f) &:= D^{r+1}(f) -\frac{k+r}{12} [E_{2} , f ]_{r}^{(2,k)} \notag \\
&= D^{r+1}(f) - \frac{k+r}{12} \sum_{i=0}^{r} (-1)^{i} \binom{r+1}{i+1} \binom{k+r-1}{i} D^{i}(E_{2}) D^{r-i}(f) . 
\end{align*}
In particular, when $r=0,1$, this operator gives the Serre derivative $\partial_{k}$ and the differential operator $L_{k+1}$ defined by \eqref{eq:dopLw}, respectively. Also, by specializing Proposition 3.3 in \cite{kaneko2006extremal}, we can see that $f \in QM_{k}^{(r)} \Rightarrow \theta_{k-r}^{(r)}(f) \in QM_{k+2(r+1)}^{(r)}$.

The following identity holds for the integer $r \in \{2,3,4\}$:
\begin{align}
\theta_{k}^{(r)} = \frac{(k+r-1)(k+2r)}{2(r-1)(k+r)} \theta_{k+2}^{(r-1)} \circ \partial_{k}   - \frac{k(k+r+1)}{2(r-1)(k+r)} \partial_{k+2r} \circ \theta_{k}^{(r-1)} . \label{eq:composition_of_theta}
\end{align}
The proof of this identity by direct calculation is long and tedious, but by using \eqref{eq:composition_of_theta} repeatedly, we can easily rewrite the differential operator $\theta_{k}^{(r)}$ into a form using the Serre derivative $\partial_{k}$ for $r \in \{2,3,4\}$. 
As will be described later, some extremal quasimodular forms are annihilated by the differential operator $\theta_{k}^{(r)}$.  Therefore, such rewriting is useful to derive an inductive structure of extremal quasimodular forms according to Grabner's method. 
Incidentally, the identity \eqref{eq:composition_of_theta} does not seem to hold for $r\ge 5$, for example, we have the following identities for $r=5$ and $6$:
\begin{align*}
&{}\theta_{k}^{(5)} - \left\{ \frac{(k+4)(k+10)}{8(k+5)} \theta_{k+2}^{(4)} \circ \partial_{k}  -  \frac{k(k+6)}{8(k+5)} \partial_{k+10} \circ \theta_{k}^{(4)} \right\} \\
&= \frac{k(k+4)(k+6)(k+10)(k^{2}+10k +36)}{1440} \Delta , \\
&{}\theta_{k}^{(6)} - \left\{ \frac{(k+5)(k+12)}{10(k+6)} \theta_{k+2}^{(5)} \circ \partial_{k}  -  \frac{k(k+7)}{10(k+6)} \partial_{k+12} \circ \theta_{k}^{(5)} \right\} \\
&= \frac{k(k+5)(k+7)(k+12)(k^{2}+12k +47)}{300} \Delta \partial_{k}.
\end{align*}

In \cite{kaneko2006extremal} Kaneko and Koike showed that $\theta_{w-1}^{(1)}(G_{w}^{(1)})=0 \; (w\equiv 0 \pmod{6})$ (Recall the Kaneko--Zagier equation \eqref{eq:2ndKZeqn}.) and $\theta_{w-2}^{(2)}(G_{w}^{(2)})=0 \; (w\equiv 0 \pmod{4})$, and conjectured that $\theta_{w-3}^{(3)}(G_{w}^{(3)})=0 \; (w\equiv 0 \pmod{6})$ and $\theta_{w-4}^{(4)}(G_{w}^{(4)})=0 \; (w\equiv 0 \pmod{12})$. Hence, for $r \in \{1,2,3,4\}$, we also call the differential equation $\theta_{w-r}^{(r)}(f)=0$ the Kaneko--Zagier equation. 
Recently, Grabner \cite{grabner2020quasimodular} gave an affirmative answer for their conjecture, more precisely, he gave a concrete form of the differential equation satisfied by the balanced quasimodular forms. 
By specializing his results, we can find the concrete form of the differential equation satisfied by the normalized extremal quasimodular forms $G_{w}^{(r)}$ for any even integer $w\ge0$ and $r\in\{1,2,3\}$ and $G_{12n}^{(4)}$. 
Of course, for $w\equiv0 \pmod{12}$ and $a \in \{2,4,6,8,10\}$, there is a fifth-order differential operator $L_{w,a}^{(4)}$ such that $L_{w,a}^{(4)}(G_{w+a}^{(4)}) =0$. 
The specific form of $L_{w,a}^{(4)}$ is not described here, but is of the form shown below:
\begin{align*}
L_{w,a}^{(4)} = f_{a}\, \partial_{w+a-4}^{5} + (\text{lower-order terms on $\partial_{*}$}) \in M_{*}(\Gamma)[\partial_{*}], 
\end{align*}
where $(a,f_{a}) = (2,E_{10}), (4,E_{8}), (6,E_{6}), (8,E_{4}),  (10,E_{14})$. 

%%%%%%%%%%%%%%%%%%%%%%%%%%%%%%%%%%%%%%%%%%

\subsection{Case of depth $r\ge2$}
\textbf{Depth 2.} The Kaneko--Zagier equation $\theta_{w-2}^{(2)}(f)=0$ is transformed into
\begin{align*}
&{} z^{2}(1-z) \frac{d^{3} g}{d z^{3}} + z \left( -\frac{w-12}{4}+ \frac{w-18}{4} z \right) \frac{d^{2} g}{d z^{2}} \\
&{} \quad + \left( -\frac{w-4}{4} -\frac{3w^{2} -72w +452}{144} z \right) \frac{d g}{d z} +\frac{(w-2)(w-6)(w-10)}{1728} g = 0
\end{align*}
or equivalently
\begin{align*}
\left\{ \Theta^{2} \left( \Theta - \frac{w-4}{4} -1  \right) - t \left( \Theta - \frac{w-2}{12} \right) \left( \Theta - \frac{w-6}{12} \right) \left( \Theta - \frac{w-10}{12} \right)\right\} g =0,
\end{align*}
where $g(\tau ) = E_{4}(\tau )^{-(w-2)/4} f(\tau )$ and $z=1728/j(\tau )$. Since this differential equation is a hypergeometric differential equation, a similar statement as in Proposition \ref{prop:Gto2F1} holds for depth $2$. 
\begin{proposition}
The normalized extremal quasimodular forms of even weight and depth $2$ on $\Gamma$ have the following hypergeometric expressions.
\begin{align}
\begin{split}
G_{4n}^{(2)}(\tau ) &= j(\tau )^{-n} E_{4}(\tau )^{(2n-1)/2}\\
&{} \quad \times {}_{3}F_{2} \left( \frac{4n+1}{6} , \frac{4n+3}{6} ,  \frac{4n+5}{6} ; n+1, n+1 ; \frac{1728}{j(\tau )} \right), 
\end{split}\label{eq:d2G4nHyp} \\
\begin{split}
G_{4n+2}^{(2)}(\tau ) &= j(\tau )^{-n} E_{4}(\tau )^{(2n-3)/2} E_{6}(\tau )\\
&{} \quad \times {}_{3}F_{2} \left( \frac{4n+3}{6} , \frac{4n+5}{6} ,  \frac{4n+7}{6} ; n+1, n+1 ; \frac{1728}{j(\tau )} \right) . 
\end{split}\label{eq:d2G4n+2Hyp}
\end{align}
\end{proposition}
From Grabner's result \cite{grabner2020quasimodular}, the relation $G_{4n+2}^{(2)} = \frac{6}{4n+1} \partial_{4n-2}(G_{4n}^{(2)})$ holds, so the latter expression can be obtained from the former after a little calculation. 
We omit the proof because it is similar to Proposition~\ref{prop:Gto2F1}. 
It is known from \cite{kaminaka2021extremal} that $\mathcal{E}_{2}=\{4,8\}$, and the Fourier coefficients of the corresponding normalized extremal quasimodular forms are given below:  
\begin{align*}
G_{4}^{(2)} &= \frac{E_{4} - E_{2}^{2}}{288} = - \frac{D(E_{2})}{24} = \sum_{n=1}^{\infty}  n \sigma_{1}(n) q^{n} \\
&= q+6 q^2+12 q^3+28 q^4+30 q^5+72 q^6+56 q^7+120 q^8+O(q^{9})   , \\
G_{8}^{(2)} &= \frac{5 E_{4}^{2} + 2 E_{2} E_{6} -7 E_{2}^{2} E_{4}}{362880} = \frac{1}{30} \sum_{n=2}^{\infty}  n( \sigma_{5}(n) - n \sigma_{3}(n) ) q^{n} \\
&= q^2+16 q^3+102 q^4+416 q^5+1308 q^6+3360 q^7+7772 q^8+O(q^{9}).
\end{align*}
The positivity of the Fourier coefficients of $G_{8}^{(2)}$ can also be seen from the fact that $D(G_{8}^{(2)})=2G_{6}^{(1)}G_{4}^{(2)}$. 
In these cases, it is easier to show the integrality by calculating the Fourier coefficients directly, but here we introduce a proof using a hypergeometric series as in the case of depth~1. We define
\begin{align*}
R_{n}(t) = {}_{3}F_{2} \left( \frac{4n+1}{6} , \frac{4n+3}{6} ,  \frac{4n+5}{6} ; n+1, n+1 ; 1728 t \right),
\end{align*}
and then $R_{1}(t) \in \mathbb{Z}[\![t]\!] \Leftrightarrow G_{4}^{(2)} \in \mathbb{Z}[\![q]\!]$ and $R_{2}(t) \in \mathbb{Z}[\![t]\!] \Leftrightarrow G_{8}^{(2)} \in \mathbb{Z}[\![q]\!]$ hold. Since
\begin{align*}
R_{1}(t) &= \sum_{r=0}^{\infty} \binom{2r+1}{r} \binom{3r+1}{r+1} \binom{6r+1}{3r}\, t^{r}, \\
R_{2}(t) &= \sum_{r=0}^{\infty} \frac{1}{210} \binom{2r+2}{r} \binom{3r+4}{r+2} \binom{6r+7}{3r+3}\, t^{r},
\end{align*}
the integrality of the former is obvious, while that of the latter is not. 
Noticing that the exponents of the prime factors of $210=2 \cdot 3 \cdot 5 \cdot  7$ are all $1$, we can prove the latter integrality using the Lucas' theorem in \eqref{eq:Lucas}. For $p=5,7$ we set $r=\ell p+m\; (\ell \in \mathbb{Z}_{\ge0}, 0\le m \le p-1)$ and use Lucas' theorem only once, but for $p=2,3$, we have to use it several times. For example, the calculation for $p=2$ and $r=8\ell+2$ is performed as follows.
\begin{align*}
&{} \binom{16\ell+6}{8\ell+2} \binom{24\ell+10}{8\ell+4} \binom{48\ell+19}{24\ell+9} \equiv \binom{0}{0}^{2} \binom{1}{1} \binom{8\ell+3}{4\ell+1} \binom{12\ell+5}{4\ell+2} \binom{24\ell+9}{12\ell+4} \\
&\equiv \binom{1}{1} \binom{1}{0}^{2} \binom{4\ell+1}{2\ell} \binom{6\ell+2}{2\ell+1}  \binom{12\ell+4}{6\ell+2} \equiv \binom{1}{0} \binom{0}{1} \binom{0}{0} \binom{2\ell}{\ell} \binom{3\ell+1}{\ell} \binom{6\ell+2}{3\ell +1} \\
&\equiv 0 \pmod{2} \; \left( \text{note } \tbinom{0}{1}=0 \right).
\end{align*}

The following assertions are the depth $2$ counterparts of Proposition \ref{prop:d1basisB}. The proof is the same, so it is omitted. 
\begin{proposition}
For each even integer $k\ge2$, a basis $\mathcal{B}_{k}^{(2)}$ of the space $QM_{k}^{(2)}(\Gamma)$ is given by the following set:
\begin{align*}
\mathcal{B}_{4n}^{(2)} &= \left\{ E_{4}^{n}, E_{4}^{n-1} G_{4}^{(2)}, \cdots , E_{4} G_{4n-4}^{(2)}, G_{4n}^{(2)} \right\} , \\
\mathcal{B}_{4n+2}^{(2)} &= \left\{ E_{4}^{n}G_{2}^{(2)}, E_{4}^{n-1} G_{6}^{(2)}, \cdots , E_{4} G_{4n-2}^{(2)}, G_{4n+2}^{(2)} \right\} .
\end{align*}
\end{proposition}

\textbf{Depth 3.} The Kaneko--Zagier equation $\theta_{w-3}^{(3)}(f)=0$ is transformed into  
\begin{align*}
&{} z^{3}(1-z)^{2} \frac{d^{4} g}{d z^{4}} + z^{2} (1-z) \left( -\frac{w-18}{3} + \frac{w-27}{3} z  \right)  \frac{d^{3} g}{d z^{3}} \\
&+ z \left( -w +7 - \frac{3w^{2}-198w +1825}{72} z + \frac{3w^{2}-126w +1375}{72} z^{2} \right)  \frac{d^{2} g}{d z^{2}} \\
&+ \left\{ -\frac{w-3}{3} +\frac{w^3-36 w^2+781 w-3540}{432} z -\frac{(w-15)(w^2-30 w+241)}{432} z^{2} \right\} \frac{d g}{d z} \\
&+ \frac{(w-3) \left\{ -4 (w^3-3 w^2+62 w-180) + (w-7)(w-11)(w-15) z  \right\} }{20736} g =0,
\end{align*}
where $g(\tau ) = E_{4}(\tau )^{-\frac{w-3}{4}} f(\tau )$ and $z=1728/j(\tau )$.
This differential equation is not a hypergeometric differential equation. However, noting the ``accidental relation" $G_{6}^{(3)} = \tfrac{1}{30} ( D(G_{4}^{(2)}) - E_{2} G_{4}^{(2)} )$, we obtain the following expression of $G_{6}^{(3)}$ using a hypergeometric series:  
\begin{align*}
G_{6}^{(3)} &= j^{-2} \left( 1- \frac{1728}{j} \right)^{1/2} E_{4}^{3/4} \\
&{} \quad \times \left\{ 21\, {}_{2}F_{1} \left( \frac{1}{12}, \frac{5}{12} ; 1 ; \frac{1728}{j} \right) {}_{3}F_{2} \left( \frac{11}{6}, \frac{13}{6}, \frac{5}{2} ; 3, 3 ; \frac{1728}{j} \right) \right. \\
&{} \qquad \left. -20\,  {}_{2}F_{1} \left( \frac{13}{12}, \frac{17}{12} ; 2 ; \frac{1728}{j} \right) {}_{3}F_{2} \left( \frac{5}{6}, \frac{7}{6}, \frac{3}{2} ; 2, 2 ; \frac{1728}{j} \right) \right\} \\
&= t^{2}+1496 t^{3}+2072262 t^{4}+2893548528 t^{5} +O(t^{6}) \quad (t=1/j) .
\end{align*}

It is known from \cite{kaminaka2021extremal} that $\mathcal{E}_{3}=\{6\}$, more specifically, we have
\begin{align*}
G_{6}^{(3)} &= \frac{5E_{2}^{3} - 3 E_{2} E_{4} -2 E_{6}}{51840} = \frac{1}{6} \sum_{n=2}^{\infty}  n ( \sigma_{3}(n) - n \sigma_{1}(n) ) q^{n} \\
&= q^2+8 q^3+30 q^4+80 q^5+180 q^6+336 q^7+620 q^8+960 q^9+O(q^{10}) \in \mathbb{Z}[\![q]\!].
\end{align*}
The positivity of the Fourier coefficients of $G_{6}^{(3)}$ can also be seen from the fact that $D(G_{6}^{(3)})=2(G_{4}^{(2)})^{2}$. 
In considering the meaning of integrality of Fourier coefficients, it is worth noting that the $d$-th Fourier coefficient of $G_{6}^{(3)}$ gives the number of simply ramified coverings of genus $2$ and degree $d$ of an elliptic curve over $\mathbb{C}$. See  \cite{dijkgraaf1995mirror,kaneko1995generalizedJacobitheta} for more details on this claim.

\begin{proposition}
For any even integer $k\ge2$, a basis $\mathcal{B}_{k}^{(3)}$ of the space $QM_{k}^{(3)}(\Gamma)$ is given by the following set:
\begin{align*}
\mathcal{B}_{12m}^{(3)} &= \left\{ E_{4}^{3(m-\ell)} G_{12\ell }^{(3)} \right\}_{0\le \ell \le m} \cup \left\{ E_{4}^{3(m-\ell)-1} G_{12\ell +4}^{(3)} \right\}_{0\le \ell \le m-1} \\
&{} \quad \cup \left\{ E_{4}^{3(m-\ell)-2} G_{12\ell +8}^{(3)} \right\}_{0\le \ell \le m-1} \cup \left\{ E_{4}^{3(m-\ell)-4} E_{6} G_{12\ell +10}^{(3)} \right\}_{0\le \ell \le m-2} \\
&{} \quad \cup \left\{ E_{6} G_{12m-6}^{(3)} \right\} , \\
%%%%%%%%%%%%
\mathcal{B}_{12m+2}^{(3)} &= \left\{ E_{4}^{3(m-\ell)} G_{12\ell +2}^{(3)} \right\}_{0\le \ell \le m} \cup \left\{ E_{4}^{3(m-\ell)-2} E_{6} G_{12\ell +4}^{(3)} \right\}_{0\le \ell \le m-1} \\
&{} \quad \cup \left\{ E_{4}^{3(m-\ell)-1} G_{12\ell +6}^{(3)} \right\}_{0\le \ell \le m-1} \cup  \left\{ E_{4}^{3(m-\ell)-2} G_{12\ell +10}^{(3)} \right\}_{0\le \ell \le m-1}, \\
%%%%%%%%%%%%
\mathcal{B}_{12m+4}^{(3)} &= E_{4} \mathcal{B}_{12m}^{(3)} \cup \left\{ G_{12m +4}^{(3)} \right\} , \quad 
%%%%%%%%%%%%
\mathcal{B}_{12m+6}^{(3)} = E_{4} \mathcal{B}_{12m+2}^{(3)} \cup \left\{ E_{6} G_{12m }^{(3)} , G_{12m +6}^{(3)} \right\} , \\
%%%%%%%%%%%%
\mathcal{B}_{12m+8}^{(3)} &= E_{4} \mathcal{B}_{12m+4}^{(3)} \cup \left\{ G_{12m +8}^{(3)} \right\} , \quad 
%%%%%%%%%%%%
\mathcal{B}_{12m+10}^{(3)} = E_{4} \mathcal{B}_{12m+6}^{(3)} \cup \left\{ G_{12m +10}^{(3)} \right\} .
\end{align*}
\end{proposition}
Note that the Fourier expansion of each element of the basis $\mathcal{B}_{12m+6}^{(3)}$ is slightly different from that of depths 1 and 2. More specifically, there are two forms $E_{4} G_{12m+2}^{(3)}$ and $E_{6} G_{12m}^{(3)}$ whose Fourier expansion is $q^{4m}(1+O(q))$. Therefore, the form $f \in QM_{12m+6}^{(3)}$ such that $f = q^{4m+1}(1+O(q))$ is expressed as a linear combination of $E_{4} G_{12m+2}^{(3)},  E_{6} G_{12m}^{(3)}$, and $G_{12m+6}^{(3)} =q^{4m+2}(1+O(q))$.

\textbf{Depth 4.} The Kaneko--Zagier equation $\theta_{w-4}^{(4)}(f)=0$ is transformed into 
\small
\begin{align*}
&{} z^{4}(1-z)^{2} \frac{d^{5} g}{d z^{5}} + 5 z^{3} (1-z) \left( -\frac{w-24}{12} + \frac{w-36}{12} z  \right) \frac{d^{4} g}{d z^{4}} \\
&{} + 5 z^{2} \left( -\frac{w-10}{2} -  \frac{w^{2}-96 w+1234}{72} z + \frac{w^{2}-60 w+928}{72} z^{2}  \right)  \frac{d^{3} g}{d z^{3}} \\
&{} - 5 z \left\{ \frac{7 w -36}{12} - \frac{w^3-54 w^2+2118 w-14796}{864} z + \frac{(w-24)(w^2-48 w+624)}{864} z^{2} \right\} \frac{d^{2} g}{d z^{2}} \\
&{} + \left\{ -\frac{5 w -12}{12} -\frac{5 (w-5) (w^3-7 w^2+216 w-3600)}{5184} z \right. \\
&{} \left. \qquad +\frac{5 (w-16) (w-20)(w^2-36 w+352)}{20736}  z^{2} \right\}   \frac{d g}{d z} \\
&{} - \frac{(w-4)(w-8) \left\{ -12(w^3+2 w^2+27 w-90) + (w-12)(w-16)(w-20) z \right\}}{248832} g =0,
\end{align*}
\normalsize
where $g(\tau ) = E_{4}(\tau )^{-(w-4)/4} f(\tau )$ and $z=1728/j(\tau )$.
This differential equation is not a hypergeometric differential equation. 
It is known from \cite{kaminaka2021extremal} that $\mathcal{E}_{4}= \emptyset$, that is, there are no normalized extremal quasimodular forms with integral Fourier coefficients. However, we note that $D(G_{8}^{(4)})=3G_{4}^{(2)}G_{6}^{(3)}$ holds as for depths~2~and~3.

\begin{remark}
The fourth- and fifth-order differential equations that appear for depths $3$ and $4$, respectively, are ``rigid" in the sense of Katz (see \cite[Rem.~17]{franc2016hypergeometric}). 
\end{remark}

\begin{proposition}
For any even integer $k\ge2$, a basis $\mathcal{B}_{k}^{(4)}$ of the space $QM_{k}^{(4)}(\Gamma)$ is given by the following set:
\begin{align*}
\mathcal{B}_{12m}^{(4)} &= \left\{ E_{4}^{3(m-\ell)} G_{12\ell }^{(4)} \right\}_{0\le \ell \le m} \cup \left\{ E_{4}^{3(m-\ell)-1} G_{12\ell +4}^{(4)} \right\}_{0\le \ell \le m-1} \\
&{} \quad \cup \left\{ E_{4}^{3(m-\ell -1)} E_{6} G_{12\ell +6}^{(4)} \right\}_{0\le \ell \le m-1} \cup \left\{ E_{4}^{3(m-\ell)-2} G_{12\ell +8}^{(4)} \right\}_{0\le \ell \le m-1} \\
&{} \quad \cup \left\{ E_{4}^{3(m-\ell)-4} E_{6} G_{12\ell +10}^{(4)} \right\}_{0\le \ell \le m-2} \cup \left\{ E_{4} E_{6} G_{12m-10}^{(4)} \right\} , \\
%%%%%%%%%%%%
\mathcal{B}_{12m+2}^{(4)} &= \left\{ E_{4}^{3(m-\ell)} G_{12\ell +2}^{(4)} \right\}_{0\le \ell \le m} \cup \left\{ E_{4}^{3(m-\ell)-2} E_{6} G_{12\ell +4}^{(4)} \right\}_{0\le \ell \le m-1} \\
&{} \quad \cup \left\{ E_{4}^{3(m-\ell )-1} G_{12\ell +6}^{(4)} \right\}_{0\le \ell \le m-1}  \cup \left\{ E_{4}^{3(m-\ell -1)} E_{6} G_{12\ell +8}^{(4)} \right\}_{0\le \ell \le m-1} \\
&{} \quad \cup \left\{ E_{4}^{3(m-\ell)-2} G_{12\ell +10}^{(4)} \right\}_{0\le \ell \le m-1} , \\
%%%%%%%%%%%%
\mathcal{B}_{12m+4}^{(4)} &= E_{4} \mathcal{B}_{12m}^{(4)} \cup \left\{ G_{12m +4}^{(4)} \right\} , \\ 
%%%%%%%%%%%%
\mathcal{B}_{12m+6}^{(4)} &= E_{4} \mathcal{B}_{12m+2}^{(4)} \cup \left\{ E_{6} G_{12m}^{(4)} , G_{12m+6}^{(4)} \right\} , \\
%%%%%%%%%%%%
\mathcal{B}_{12m+8}^{(4)} &= E_{4} \mathcal{B}_{12m+4}^{(4)} \cup \left\{ E_{6} G_{12m+2}^{(4)} , G_{12m+8}^{(4)} \right\}, \\
%%%%%%%%%%%%
\mathcal{B}_{12m+10}^{(4)} &= E_{4} \mathcal{B}_{12m+6}^{(4)} \cup \left\{ G_{12m +10}^{(4)} \right\} .
\end{align*}
\end{proposition}
The same (but slightly more complicated) phenomenon as in the case of depth $3$ occurs for the Fourier expansion of each element of these bases. For instance, there is no element in the set $\mathcal{B}_{12m}^{(4)}$ for which $q^{5\ell+4}(1+O(q))\;(0\le \ell \le m-1)$, but there are two elements for which $q^{5\ell+3}(1+O(q))\;(0\le \ell \le m-2)$. However, there is only one element for which $q^{5m-2}(1+O(q))$, and there are two elements for which $q^{5m-5}(1+O(q))$.

For a given even integer $12n+a \; (n \in \mathbb{Z}_{\ge0}, a \in 2\mathbb{Z}_{\ge0})$ and depth $r\ge1$, let $12n+a-2r=4s+6t+12u \; (s \in \{0,1,2\}, t \in \{0,1\}, u \in \mathbb{Z}_{\ge 0})$ as before, and express the normalized extremal quasimodular form $G_{12n+a}^{(r)}$ as follows: 
\begin{align*}
G_{12n+a}^{(r)} = \frac{1}{N_{n,a}^{(r)}} E_{2}^{r} E_{4}^{s} E_{6}^{t} \Delta^{u} A_{u,a}^{(r)}(j) + (\text{lower order terms on $E_{2}$}) ,
\end{align*}
where the number $N_{n,a}^{(r)}$ is chosen so that the polynomial $A_{u,a}^{(r)}(X)$ of degree $u$ is the monic polynomial. We call this polynomial the generalized Atkin polynomial. 
Although it is possible to calculate the polynomial $A_{u,a}^{(r)}(X)$ separately for a given $n,a$ and $r$, as already mentioned in Section~\ref{sec:intro}, the existence and uniqueness of $G_{w}^{(r)}$ for the  depth $r\ge5$ are still open, it is not yet clear whether the polynomial $A_{u,a}^{(r)}(X)$ determined by the above equation always exists for $r\ge5$.

The following conjecture is suggested by numerical experiments . 
\begin{conjecture}
Let $p\ge5$ be a prime number and put $p-1=12m+4\delta+6\varepsilon \; (m \in \mathbb{Z}_{\ge 0}, \delta, \varepsilon \in \{0,1\})$. Then we have $A_{m+\delta+\varepsilon,2}^{(1)}(X) \equiv A_{m+\delta+\varepsilon,2r}^{(r)}(X) \pmod{p}$ for $r\ge2$. 
\end{conjecture}

\begin{example}
The case of $m+\delta+\varepsilon=2$, that is, the case of $p\in \{11,17,19\}$:
\begin{align*}
A_{2,2}^{(1)}(X) &= X^{2} -1640 X +269280 \equiv 
\begin{cases}
X(X+10) \pmod{11} \\
X(X+9) \pmod{17} \\
(X+1)(X+12) \pmod{19},
\end{cases} \\
A_{2,4}^{(2)}(X) &= X^{2} +\frac{5367564}{4847} X +\frac{97748640}{4847} ,\\
A_{2,6}^{(3)}(X) &= X^{2} -\frac{100925285400}{6736603} X +\frac{184720492440000}{6736603} ,\\
A_{2,8}^{(4)}(X) &= X^2+\frac{13326301537125}{303744733} X +\frac{8760324756150000}{303744733} , \\
A_{2,10}^{(5)}(X) &= X^{2} -\frac{567274769925055704588000}{7988288882724700441} X \\
&{} \quad +\frac{302601299124728270224800000}{7988288882724700441} , \\
A_{2,12}^{(6)}(X) &= X^2+\frac{67508245504783855161034500000}{433955868750758754759533} X \\
&{} \quad -\frac{163976620145430859347886034400000}{433955868750758754759533}. 
\end{align*}
For the prime $p\in \{11,17,19\}$, these polynomials satisfy the above conjecture.
\end{example}

For the depths greater than 5, almost nothing is known. For example, we do not even know what differential equation the extremal quasimodular form satisfies. 
Unlike the case of depth $r\le 4$, as already pointed out in \cite{kaneko2006extremal}, any extremal modular form of weight $w$ and depth $r\ge5$ cannot satisfy the differential equation $\theta_{w-r}^{(r)}(f)=0$.
Of course, it is possible to calculate specific examples separately, e.g., the normalized extremal quasimodular form of weight $10$ and depth $5$
\begin{align*}
G_{10}^{(5)} &= \frac{140 E_{2}^{5} -35 E_{2}^{3} E_{4} -65 E_{2}^{2} E_{6} -33 E_{2} E_{4}^{2} -7 E_{4}E_{6} }{1437004800} \\
&= q^{4} +\frac{144}{11}q^{5}+\frac{936}{11} q^{6} +\frac{4160}{11} q^{7} +\frac{14490}{11} q^{8} +\frac{42432}{11} q^{9} +O(q^{10})
\end{align*}
is the solution of the following modular linear differential equation:
\begin{align*}
&{} \left( E_{4}^{3} - \frac{731087}{4380623} E_{6}^{2} \right) \partial_{5}^{6}(f) + \frac{3649536}{4380623} E_{4}^{2}E_{6} \partial_{5}^{5}(f) \\
&{} - \frac{5 E_{4} (845736619 E_{4}^{3} -170572459 E_{6}^{2} )}{630809712} \partial_{5}^{4}(f) \\
&{} - \frac{5 E_{6} (2032753837 E_{4}^{3} -164191405 E_{6}^{2})}{946214568} \partial_{5}^{3}(f) \\
&{} + \frac{E_{4}^{2} (262935868013 E_{4}^{3} - 746094289517 E_{6}^{2})}{90836598528} \partial_{5}^{2}(f) \\
&{} + \frac{E_{4}E_{6} (80592093937 E_{4}^{3} - 122767956721 E_{6}^{2})}{45418299264}  \partial_{5}(f) \\
&{} + \frac{55 (3672965829 E_{4}^{6} -7414522789 E_{4}^{3} E_{6}^{2} -5174923040 E_{6}^{4})}{13080470188032} f = 0.
\end{align*}
As can be expected from this example, as the weights or depths increases, the coefficients of the differential equation become more complex and the calculations rapidly become unmanageable.
From the observation of some specific examples, it seems that the coefficients cannot be a simple polynomial with $w$ as a variable. 
We also note that the denominator 11 of the Fourier coefficients of $G_{10}^{(5)}$ is greater than the weight 10. (cf. Corollary \ref{cor:denomd1exqmf})

At the end of Section \ref{sec:Gasfpsj}, we pointed out that the normalized extremal quasimodular forms of depth 1 are essentially the remainder of some Hermite--Pad\'{e} approximation. 
In fact, with the help of the ring isomorphism in Theorem \ref{thm:isomQMF}, the extremal quasimodular forms of any depth can be regarded as the remainder of some Hermite--Pad\'{e} approximation (of the first kind) for a suitable function. 
On the other hand, there is an approximation called the Hermite--Pad\'{e} approximation of the second kind (or simultaneous Pad\'{e} approximation) \cite{nesterenko1995hermite}. 
For the depth $r=1$, these two approximations are essentially the same, but not for the depth $\ge2$. What are the properties of the quasimodular forms derived from the simultaneous Pad\'{e} approximation?

%%%%%%%%%%%%%%%%%%%%%%%%%%%%%%%%%%%%%%%%%%

\subsection{Case of other groups}
In this subsection we summarize previous works on Atkin orthogonal polynomials and extremal quasimodular forms of depth 1 on the congruence subgroup $\Gamma_{0}(N)$ and the Fricke group $\Gamma_{0}^{*}(N)$ of low-levels.
If the level is less than or equal to 4, the corresponding quasimodular Eisenstein series of weight 2 can be expressed using the hypergeometric series ${}_{2}F_{1}$ (see \cite[Ch.~3]{nakaya2018phdthesis}), so it is probably possible to treat the normalized extremal quasimodular forms in the same way as in this paper. 
On the other hand, for the Fricke groups of levels 5 and 7, the local Heun function appears instead of the hypergeometric series ${}_{2}F_{1}$. 
Because of this difference, it may be difficult to obtain various concrete expressions for these levels as in this paper. 
Note also that in the case of level 1, the Atkin polynomials and the normalized extremal quasimodular forms are directly connected as in \eqref{eq:G12m2AB}, but this is not necessarily the case for level $\ge2$. 

The next table summarizes the previous works by Tsutsumi, Kaneko, Koike, Sakai, and Shimizu. Note that some extremal quasimodular forms also appear in \cite{sakai2014atkin}, and the Kaneko--Zagier equation for $\Gamma_{0}^{*}(2)$ is treated in \cite{kaneko2004quasimodular}, but its quasimodular solution is the quasimodular form on $\Gamma_{0}(2)$. 
\begin{table}[H]
\caption{The previous works related to this paper (with levels).}
\begin{center}
\begin{tabular}{|c|c|c|} \hline
 & Atkin inner product  & Modular linear differential equation \\ 
 & Atkin polynomials & (Extremal) quasimodular forms \\ \hline
  \multirow{2}{*}{$\Gamma_{0}(N)$} & \multirow{2}{*}{Tsutsumi \cite{tsutsumi2000atkininnerproduct,tsutsumi2007atkin} $(N=2,3,4)$}  & Sakai--Tsutsumi \cite{sakaitsutsumi2012extremal} $(N=2,3,4)$ \\
 &  & Sakai--Shimizu \cite{sakai2015modular} $(N=2,3,4)$ \\ \hline 
 \multirow{2}{*}{$\Gamma_{0}^{*}(N)$} & Sakai \cite{sakai2014atkin} $(N=5,7)$ &  Kaneko--Koike \cite{kaneko2004quasimodular} $(N=2)$ \\ \cline{2-3}
 & \multicolumn{2}{c|}{ Sakai \cite{sakai2011atkin} $(N=2,3)$   } \\ \hline
\end{tabular}
\end{center}
\end{table}
Here we mention a few problems.
\begin{itemize}
\item In the case of level 1, three differential equations appeared in \cite[Thm.~2.1]{kaneko2006extremal}, but since there are relations \eqref{eq:recGw0} and \eqref{eq:recGw24}, it is essentially sufficient to consider only the differential equation \eqref{eq:2ndKZeqn}. Extend this result to level $\ge2$ and clarify the relationship between quasimodular solutions of some differential equations using Grabner's method.  
\item Extend Theorem \ref{thm:mainthm} to level $\ge2$ and depth $\ge 1$. 
\end{itemize}

Although of little relevance to quasimodular forms, it is worth noting that a certain basis of the  vector space of modular forms of weight $2n/7\;(n \in \mathbb{Z}_{\ge1})$ on $\Gamma(7)$ is expressed by a generalized hypergeometric series ${}_{3}F_{2}$. For more details on this fact, see \cite{ibukiyama2000modular,franc2016hypergeometric,vidunas2020darboux} and its references. 

\section*{Acknowledgement}
The author thanks to Professor Masanobu Kaneko for many valuable comments on reading an early draft of this paper.

%%%%%%%%%%%%%%%%%%%%%%%%%%%%%%%%%%%%%%%%%%%%%%%%%%%%%%%%%%%%%%%%%%%%%%%%%%%%%%%%%%%%%%%%%
%%%%%%%%%%%%%%%%%%%%%%%%%%%%%%%%%%%%%%%%%%%%%%%%%%%%%%%%%%%%%%%%%%%%%%%%%%%%%%%%%%%%%%%%%
%%%%%%%%%%%%%%%%%%%%%%%%%%%%%%%%%%%%%%%%%%%%%%%%%%%%%%%%%%%%%%%%%%%%%%%%%%%%%%%%%%%%%%%%%

\appendix
\section{Appendices}
\subsection{Table of the integral Fourier coefficients of $G_{w}^{(1)}$}
Put $G_{w}^{(1)} = q^{[w/6]} \left( 1 + \sum_{n=1}^{\infty} a_{w}(n)\, q^{n} \right)$ and then the first few terms of the integral Fourier coefficients of $G_{w}^{(1)}$ are given in the following table. 
\begin{table}[H]
\caption{The first few integral Fourier coefficients $a_{w}(n)$}
\begin{center}
\begin{tabular}{|c|rrrrr|} \hline
 $w\backslash n$ & 1 & 2 & 3 & 4 & 5 \\ \hline
 2 & $-24$ & $-72$ & $-96$ & $-168$ & $-144$   \\
 6 & 18 & 84 & 292 & 630 & 1512   \\
 8 & 66 & 732 & 4228 & 15630 & 48312   \\
 10 & 258 & 6564 & 66052 & 390630 & 1693512  \\
 12 & 56 & 1002 & 9296 & 57708 & 269040   \\
 14 & 128 & 4050 & 58880 & 525300 & 3338496  \\
 16 & 296 & 16602 & 377456 & 4846908 & 41943120  \\
 18 & 99 & 3510 & 64944 & 764874 & 6478758  \\
 20 & 183 & 10134 & 269832 & 4326546 & 47862918  \\
 22 & 339 & 29430 & 1127904 & 24615834 & 355679478  \\
 24 & 144 & 7944 & 235840 & 4451130 & 59405952  \\
 28 & 384 & 44664 & 2460160 & 79196970 & 1693028352  \\
 30 & 190 & 14460 & 608570 & 16463120 & 314562708  \\
 32 & 286 & 29988 & 1652834 & 56608952 & 1335336084  \\
 34 & 430 & 62220 & 4496090 & 195047840 & 5680752948  \\
 38 & 336 & 43587 & 3065648 & 136437750 & 4219436160  \\
 54 & 378 & 62532 & 6109740 & 401161950 & 19083824856  \\
 58 & 618 & 155412 & 21940620 & 2005126350 & 128986599096  \\
 68 & 581 & 147042 & 21956168 & 2203554570 & 160242315903  \\
 80 & 678 & 204756 & 37135249 & 4592036697 & 416237464122  \\
 114 & 855 & 341886 & 85507600 & 15092041050 & 2010698806050  \\
 118 & 1095 & 549246 & 169413760 & 36358101930 & 5819797557810  \\ 
\hline
\end{tabular}
\end{center}
\end{table}

%%%%%%%%%%%%%%%%%%%%%%%%%%%%%%%%%%%%%%%%%%

\subsection{Explicit formulas for the coefficients of $P_{n}(t)$ and $Q_{n}(t)$}
We give explicit formulas for the coefficients of the power series $P_{n}(t)$ and $Q_{n}(t)$ that appear in the proof of Theorem \ref{thm:mainthm}. However, as we have already seen in equations \eqref{eq:Q2mAVBU} and \eqref{eq:P2m1AVBU}, we multiply the power series $P_{\text{odd}}(t)$ and  $Q_{\text{even}}(t)$ by the factor $(1-1728t)^{-1/2}$. 
Our main theorem is equivalent to the fact that the following power series belong to $\mathbb{Z}[\![t]\!]$.

$\bullet$ $G_{w}^{(1)} \in \mathbb{Z}[\![q]\!], \, w \in \mathcal{S}_{0} = \{ 12,24 \}$. The coefficients of the power series $P_{0}(t)$ and $P_{2}(t)$ correspond to the sequences A001421 and A145493 in the Online Encyclopedia of Integer Sequences (OEIS) \cite{oeis}, respectively.
\begin{align*}
P_{0}(t) &= \sum_{r=0}^{\infty} \frac{(6r)!}{(3r)!\, r!^{3}} \, t^{r}  = {}_{3}F_{2} \left( \frac{1}{6}, \frac{1}{2}, \frac{5}{6} ; 1, 1 ; 1728 t  \right) \quad ( \Leftrightarrow G_{0}^{(1)} =1 ) \\
&= 1+120 t+83160 t^2+81681600 t^3+93699005400 t^4+O(t^5) , \\
P_{2}(t) &=  \sum_{r=0}^{\infty} \frac{(41r + 77)\,(6r+6)!}{2310\, (3r+3)!\, r!\, (r+2)!^{2}} \, t^{r} \\
&= 1+944 t+1054170 t^2+1297994880 t^3+1700941165560 t^4+O(t^5) , \\
P_{4}(t) &= \sum_{r=0}^{\infty} \frac{(17377 r^{2} +117219 r+193154)\,(6r+12)!}{223092870\, (3r+6)!\, r!\, (r+4)!^{2}}  \, t^{r} \\
&= 1+1800 t+2783760 t^2+4183182720 t^3+6274984354650 t^4 +O(t^5).
\end{align*}

$\bullet$ $G_{w}^{(1)} \in \mathbb{Z}[\![q]\!], \, w \in \mathcal{S}_{6} = \{ 6,18,30,54,114 \}$.
\begin{align*}
&{} (1-1728t)^{-1/2} P_{1}(t) = \sum_{r=0}^{\infty} \frac{(6r+6)!}{120\, (3r+3)!\, r!\, (r+1)!^{2}} \, t^{r} \\
&{} \quad = {}_{3}F_{2} \left( \frac{7}{6}, \frac{3}{2}, \frac{11}{6} ; 2, 2 ; 1728 t  \right) \\
&{} \quad = 1+1386 t+2042040 t^2+3123300180 t^3+4891088081880 t^4  +O(t^5), \\
&{} (1-1728t)^{-1/2} P_{3}(t) =  \sum_{r=0}^{\infty} \frac{(77r + 221)\,(6r+12)!}{4084080\, (3r+6)!\, r!\, (r+3)!^{2}} \, t^{r}  \\
&{} \quad = 1+2235 t+4129650 t^2+7217526960 t^3+12344776903800 t^4 +O(t^5) , \\
&{} (1-1728t)^{-1/2} P_{5}(t) = \sum_{r=0}^{\infty} \frac{(33649 r^{2}+294051 r+633650)\,(6r+18)!}{776363187600\, (3r+9)!\, r!\, (r+5)!^{2}} \, t^{r}   \\
&{} \quad = 1+3094 t+6975504 t^2+13953546090 t^3+26319290241530 t^4 +O(t^5) , \\
&{} (1-1728t)^{-1/2} P_{9}(t) = \sum_{r=0}^{\infty} \frac{p_{9}(r)\, (6r+30)!}{(3r+15)! \, r! \, (r+9)!^{2}}\, t^{r}   \\
&{} \quad = 1+4818 t+14913288 t^2+37889152860 t^3+86182007602320 t^4 +O(t^5) , \\
&{} (1-1728t)^{-1/2} P_{19}(t) =  \sum_{r=0}^{\infty} \frac{p_{19}(r)\, (6r+60)!}{(3r+30)! \, r! \, (r+19)!^{2}}\, t^{r} \\
&{} \quad = 1+9135 t+47828730 t^2+188818914000 t^3+625280243661000 t^4 +O(t^5) , 
\end{align*}
where
\begin{align*}
p_{9}(r) &= \frac{1}{1303566339087601789200} \\
&{} \times (301163357 r^4+8876894690 r^3+97346883895 r^2 \\
&{} \quad +470641033450 r+846250112568), \\
p_{19}(r) &= \frac{1}{7586413113700225869154849509970478998385924877600} \\
&{} \times ( 116055861444395385601913 r^9 \\
&{} \quad +15530138946748752922984725 r^8 \\ 
&{} \quad +920111315629981006299003510 r^7 \\
&{} \quad +31676880792353832401375777850 r^6 \\
&{} \quad +698329420677409164956468289249 r^5 \\
&{} \quad +10222801871323855615909703388405 r^4 \\
&{} \quad +99369498641304011775924341700640 r^3 \\
&{} \quad +618440343527755839046417085216700 r^2 \\
&{} \quad +2236089229125717720580535903583888 r \\
&{} \quad +3578581860690243122001381266421120 ) 
\end{align*}
and these polynomials are irreducible over $\mathbb{Q}$.

$\bullet$ $G_{w}^{(1)} \in \mathbb{Z}[\![q]\!], \, w \in \mathcal{S}_{2} = \{ 2,14,38 \}$.
\begin{align*}
&{} (1-1728t)^{-1/2} Q_{0}(t) = \sum_{r=0}^{\infty} \frac{(6r+1)!}{(3r)!\, r!^{3}} \, t^{r}  = {}_{3}F_{2} \left( \frac{1}{2}, \frac{5}{6}, \frac{7}{6} ; 1, 1 ; 1728 t  \right) \\
&{} \quad = 1+840 t+1081080 t^2+1551950400 t^3+2342475135000 t^4 +O(t^5)  , \\
&{} (1-1728t)^{-1/2} Q_{2}(t) =  \sum_{r=0}^{\infty} \frac{(7r + 13)\,(6r+7)!}{2730\, (3r+3)!\, r!\, (r+2)!^{2}} \, t^{r}  \\
&{} \quad = 1+1760 t+2877930 t^2+4667789280 t^3+7590443164920 t^4 +O(t^5)  , \\
&{} (1-1728t)^{-1/2} Q_{6}(t) \\
&{} \quad = \sum_{r=0}^{\infty} \frac{(1043119 r^3+15220608 r^2+72947639 r+114757350)\, (6r+19)! }{74207381348100\, (3r+9)! \, r! \, (r+6)!^{2} } \, t^{r} \\
&{} \quad = 1+3504 t+8597259 t^2+18287498240 t^3+36144224452050 t^4 +O(t^5) . 
\end{align*}

$\bullet$ $G_{w}^{(1)} \in \mathbb{Z}[\![q]\!], \, w \in \mathcal{S}_{8} = \{ 8,20,32,68,80 \}$. The coefficients of the power series $Q_{1}(t)$ and $Q_{3}(t)$ correspond to the sequences A145492 and A145494 in OEIS, respectively.
\begin{align*}
Q_{1}(t) &= \sum_{r=0}^{\infty} \frac{(8r+7)\,(6r+1)!}{7 (3r)!\, r!\, (r+1)!^{2}} \, t^{r}  \\
&= 1+450 t+394680 t^2+429557700 t^3+522037315800 t^4 +O(t^5) , \\
Q_{3}(t) &=  \sum_{r=0}^{\infty} \frac{(1528 r^{2} +7231 r +8151)\,(6r+7)!}{190190\, (3r+3)!\, r!\, (r+3)!^{2}}  \, t^{r}  \\
&= 1+1335 t+1757970 t^2+2386445040 t^3+3336565609080 t^4 +O(t^5)  , \\
Q_{5}(t) &=  \sum_{r=0}^{\infty} \frac{(1070744 r^3+12418991 r^2+46901365 r+57574750 ) \, (6r+13)!}{34579394850\, (3r+6)! \, r! \, (r+5)!^{2}} \, t^{r}  \\
&= 1+2206 t+3863952 t^2+6319180098 t^3+10079991804410 t^4 +O(t^5)  , \\
Q_{11}(t) &= \sum_{r=0}^{\infty} \frac{q_{11}(r) \, (6r+31)!}{(3r+15)! \, r! \, (r+11)!^{2}} \, t^{r} \\
&= 1+4805 t+14658030 t^2+36441948000 t^3+80761720666320 t^4 +O(t^5)  , \\
Q_{13}(t) &= \sum_{r=0}^{\infty} \frac{q_{13}(r) \, (6r+37)!}{(3r+18)! \, r! \, (r+13)!^{2}} \, t^{r}  \\
&= 1+5670 t+19748832 t^2+54741797937 t^3+132878837538099 t^4 +O(t^5)  ,
\end{align*}
where
\begin{align*}
q_{11}(r) &= \frac{1}{15716643102160534111758180} \\
&{} \times (13252649705176 r^6+665298552506263 r^5 \\
&{} \quad +13797873461407945 r^4 +151287554887490515 r^3  \\
&{} \quad +924734694751472239 r^2 +2986992686186751022 r \\ 
&{} \quad +3982438425105968520) , \\
q_{13}(r) &= \frac{1}{32176447673406729078990845541300} \\
&{} \times (74198322973160504 r^7+5124808625350611463 r^6 \\
&{} \quad +150642927750066254963 r^5 +2442571823969345600665 r^4 \\
&{} \quad +23590276457107577780801 r^3 +135688184492311416306712 r^2  \\
&{} \quad +430315970858396108150652 r +580367220881648001413040)
\end{align*}
and these polynomials are irreducible over $\mathbb{Q}$. 

%%%%%%%%%%%%%%%%%%%%%%%%%%%%%%%%%%%%%%%%%%%%%%%%%%%%%%%%%%%%%%%%%%%%%%%%%%%%%%%%%%%%%%%%%

\bibliographystyle{abbrv}
\bibliography{ref-nakaya2023exqmfd1int.bib}

%%%%%%%%%%%%%%%%%%%%%%%%%%%%%%%%%%%%%%%%%%%%%%%%%%%%%%%%%%%%%%%%%%%%%%%%%%%%%%%%%%%%%%%%%

\end{document}